\documentclass{amsart}

\usepackage{mathrsfs}
\usepackage{amsthm}
\usepackage{amscd}
\usepackage{amssymb}
\usepackage{latexsym}
\usepackage{graphicx}
\usepackage{stmaryrd}
\usepackage{xypic}
\xyoption{curve}
\usepackage{nicefrac}
\usepackage{wasysym}

\newcommand{\ra}{\rightarrow}
\newcommand{\xra}{\xrightarrow}
\newcommand{\ira}{\hookrightarrow}

\newcommand{\id}{{\mathop{id}\;\!\!}}

\newcommand{\OP}[1]{#1^{\mathrm{op}}}

\newcommand{\BOX}{\mathbb{\square}}

\newcommand{\DEL}{\Delta}

\newcommand{\DIAGRAM}{\mathscr{G}}
\newcommand{\GENERIC}{\mathscr{C}}
\newcommand{\HOMOTOPY}{\mathop{h}}

\newcommand{\SEMILATTICES}{\mathscr{L}\;\!\!}
\newcommand{\PREORDEREDSETS}{\mathscr{Q}}
\newcommand{\PRESHEAVES}[1]{\hat{#1}}
\newcommand{\SETS}{\mathop{Set}\;\!\!}
\newcommand{\STREAMS}{\mathscr{S}}
\newcommand{\SPACES}{\mathscr{T}}

\newcommand{\I}{\mathbb{I}}

\newcommand{\graph}[1]{\mathop{graph}(#1)}
\newcommand{\join}{\vee}
\newcommand{\meet}{\wedge}

\newcommand{\Star}{\mathop{Star}\;\!\!}
\newcommand{\openstar}{\mathop{star}\;\!\!}

\newcommand{\cosimplicial}{\nabla}

\newcommand{\half}{\nicefrac{1}{2}}

\newcommand{\R}{\mathbb{R}}

\newtheorem{thm}{Theorem}[section]
\newtheorem{cor}[thm]{Corollary}
\newtheorem{lem}[thm]{Lemma}
\newtheorem{prop}[thm]{Proposition}
\newtheorem*{lem*}{Lemma}

\theoremstyle{definition}
\newtheorem{defn}[thm]{Definition}
\newtheorem{eg}[thm]{Example}

\newcommand{\boxobj}[1]{[1]^{\otimes #1}}
\newcommand{\xboxobj}[2]{[#1]^{\otimes #2}}

\newcommand{\XBOX}{\boxplus}

\newcommand{\supp}{\mathop{supp}\;\!\!}
\newcommand{\sd}{\mathop{sd}\;\!\!}
\newcommand{\cd}{\mathop{sd}\;\!\!}

\newcommand{\tri}{\mathop{tri}\;\!\!}
\newcommand{\qua}{\mathop{qua}\!}
\newcommand{\qu}{\mathop{q}\!}
\newcommand{\tr}{\mathop{t}\;\!\!}
\newcommand{\sn}{\mathop{ner}\!^\DEL}
\newcommand{\cn}{\mathop{ner}\!^\BOX}
\newcommand{\double}{\mathfrak{sd}\!}

\newcommand{\TRUNCATED}[2]{{#2}_{#1}}
\newcommand{\POSPACES}{\mathscr{P}}
\newcommand{\EQUISTREAMS}[1]{\STREAMS^{#1}}

\newcommand{\direalize}[1]{\upharpoonleft\!\!{#1}\!\!\downharpoonright}
\newcommand{\ordinomorphism}{\varphi}
\newcommand{\shomeo}{\varphi}
\newcommand{\chomeo}{\varphi}

\newcommand{\futurehomotopic}{\rightsquigarrow}
\newcommand{\homotopic}{\leftrightsquigarrow}

\newtheorem*{thm:main}{Theorem \ref{thm:main}}
\newtheorem*{cor:main}{Corollary \ref{cor:main}}
\newtheorem*{cor:excision}{Corollary \ref{cor:excision}}
\newtheorem*{thm:direalization.preserves.embeddings}{Theorem \ref{thm:direalization.preserves.embeddings}}
\newtheorem*{thm:direalization.preserves.finite.products}{Theorem \ref{thm:direalization.preserves.finite.products}}

\title{Cubical approximation for directed topology I.}
\author{Sanjeevi Krishnan}
\begin{document}
\begin{abstract}
Topological spaces - such as classifying spaces, configuration spaces and spacetimes - often admit directionality.
Qualitative invariants on such directed spaces often are more informative, yet more difficult, to calculate than classical homotopy invariants because directed spaces rarely decompose as homotopy colimits of simpler directed spaces.
Directed spaces often arise as geometric realizations of simplicial sets and cubical sets equipped with order-theoretic structure encoding the orientations of simplices and $1$-cubes.
We prove dual simplicial and cubical approximation theorems appropriate for the directed setting and give criteria for two different homotopy relations on directed maps in the literature to coincide.
\end{abstract}
\maketitle
\tableofcontents
\section{Background}
Spaces in nature often come equipped with directionality.
Topological examples of such spaces include spacetimes and classifying spaces of categories.  
Combinatorial examples of such spaces include higher categories, cubical sets, and simplicial sets.  
\textit{Directed geometric realizations} translate from the combinatorial to the topological.  
Those properties invariant under deformations respecting directionality can reveal some qualitative features of spacetimes, categories, and computational processes undetectable by classical homotopy types \cite{fgr:ditop, low:path.topology, pratt:models}.
Examples of such properties include the global orderability of spacetime events and the existence of non-determinism in the execution of concurrent programs \cite{fgr:ditop}.  
A directed analogue of singular (co)homology \cite{grandis:H}, constructed in terms of appropriate \textit{singular cubical sets} on \textit{directed spaces}, should systemize the analyses of seemingly disparate dynamical processes.   
However, the literature lacks general tools for computing such invariants.
In particular, homotopy extension properties, convenient for proving cellular and simplicial approximation theorems, almost never hold for directed maps [Figure \ref{fig:hep.failure}].  

\begin{figure}[h]
  \begin{tabular}{c}
    \includegraphics[width=30mm,height=30mm]{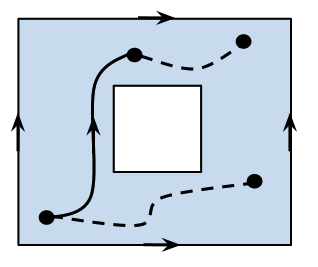} 
  \end{tabular}
  \caption{Failure of a homotopy through monotone maps to extend.  A \textit{directed path} on the illustrated square annulus $[0,3]\setminus[1,2]$ is a path monotone in both coordinates.  The illustrated dotted homotopy of maps from $\{0,1\}$ to $[0,3]^2\setminus[1,2]^2$ fails to extend to a homotopy \textit{through directed paths} from the illustrated solid directed path.}
  \label{fig:hep.failure}
\end{figure}

Current tools in the literature focus on decomposing the structure of \textit{directed paths} and undirected homotopies through such directed paths on a directed topological space.  
For general directed topological spaces, such tools include van Kampen Theorems for directed analogues of path-components \cite[Proposition 4]{fgrh:components2} and directed analogues of fundamental groupoids \cite[Theorem 3.6]{grandis:d}.
For directed geometric realizations of cubical sets, additional tools include a cellular approximation theorem for directed paths \cite[Theorem 4.1]{fajstrup:approx}, a cellular approximation theorem for undirected homotopies through directed paths \cite[Theorem 5.1]{fajstrup:approx}, and prod-simplicial approximations of undirected spaces of directed paths \cite[Theorem 3.5]{raussen:trace.spaces}.  
Extensions of such tools for higher dimensional analogues of directed paths are currently lacking in the literature.  

\section{Introduction}
Just as classical simplicial approximation \cite{curtis:approx} makes the calculation of singular (co)homology groups on compact triangulable spaces tractable, simplicial and cubical approximation theorems for streams should make the calculation of directed (co)homology theories \cite{grandis:H} equally tractable.  
The goal of this paper is to prove such approximation theorems.

We recall models, both topological and combinatorial, of directed spaces and constructions between them.
A category $\STREAMS$ of \textit{streams} \cite{krishnan:convenient}, spaces equipped with coherent preorderings of their open subsets, provides topological models. 
Some natural examples are spacetimes and connected and compact Hausdorff topological lattices.  
A category $\PRESHEAVES\BOX$ of cubical sets \cite{gl:sites} provides combinatorial models.  
The category $\PRESHEAVES\DEL$ of simplicial sets provides models intermediate in rigidity and hence serves as an ideal setting for comparing the formalisms of streams and cubical sets.    
\textit{Stream realization} functors $\direalize{-}$ [Definitions \ref{defn:simplicial.direalizations} and \ref{defn:cubical.direalizations}], \textit{triangulation} $\tri$ [Definition \ref{defn:tri}], \textit{edgewise} (\textit{ordinal}) subdivision $\sd$ [Figures \ref{fig:sd.cd.tri} and \ref{fig:sd}, \cite{ep:ordinal}, and Definition \ref{defn:sd}], and a cubical analogue $\cd$ [Definition \ref{defn:cd}] relate these three categories in the following commutative diagram [Figure \ref{fig:sd.cd.tri} and Propositions \ref{prop:sd.cd.tri} and \ref{prop:tri.direalize}]. 
\begin{equation*}
  \xymatrix{
    \PRESHEAVES\BOX\ar[dr]|{\direalize{\;-\;}}\ar@/_2ex/[dd]_{\cd}
  &
  & \PRESHEAVES\DEL\ar[dl]|{\direalize{\;-\;}}\ar@/^2ex/[dd]^{\sd}
  \\
  & \STREAMS
  \\
    \PRESHEAVES\BOX\ar[ur]|{\direalize{\;-\;}}\ar@/_2ex/_{\tri}[rr]
  &
  & \PRESHEAVES\DEL\ar[ul]|{\direalize{\;-\;}}
  }
\end{equation*}

We prove that the functors exhibit convenient properties, which we use in our proof of the main results.
Just as double barycentric simplicial subdivisions factor through polyhedral complexes \cite{curtis:approx}, quadruple cubical subdivisions locally factor, in a certain sense [Lemmas \ref{lem:cd.fold} and \ref{lem:cd.star.collapse}], through representable cubical sets.    
Triangulation and geometric realization both translate rigid models of spaces (cubical sets, simplicial sets) into more flexible models (simplicial sets, topological spaces.)  
However, the composite of triangulation with its right adjoint - unlike the composite of geometric realization with its right adjoint - is cocontinuous [Lemma \ref{lem:qt.cocontinuous}].  
Stream realization functors inherit convenient properties from their classical counterparts.

\begin{thm:direalization.preserves.finite.products}
  The functor $\direalize{-}\;:\hat\DEL\ra\STREAMS$ preserves finite products.
\end{thm:direalization.preserves.finite.products}

\begin{thm:direalization.preserves.embeddings}
  The functor $\direalize{-}\;:\PRESHEAVES\BOX\ra\STREAMS$ sends monics to stream embeddings.
\end{thm:direalization.preserves.embeddings}

\begin{figure}[t]
  \begin{tabular}{ccc}
    \includegraphics[width=20mm,height=20mm]{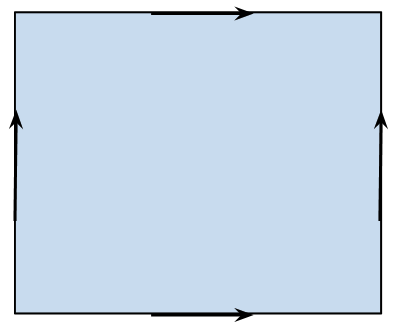} 
  &  
  & \includegraphics[width=20mm,height=20mm]{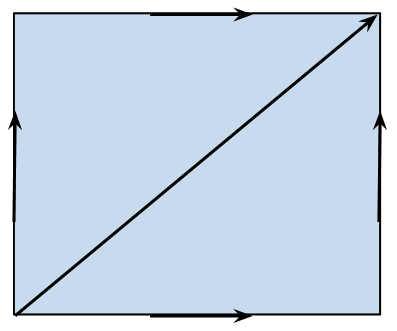}\\
    $\BOX[1]$
  &
  & $\tri\;\BOX[1]=(\DEL[1])^2$\\
  \vspace{.1cm}\\
    \includegraphics[width=20mm,height=20mm]{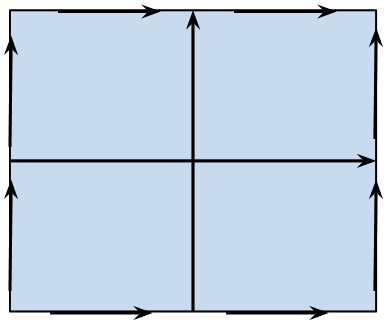} 
  &  
  & \includegraphics[width=20mm,height=20mm]{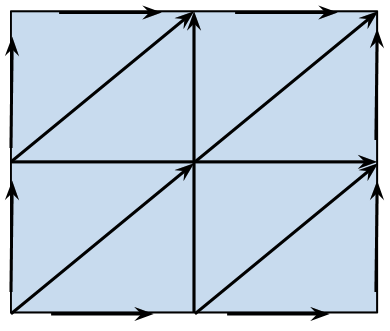}\\ 
    $\cd\BOX[1]$
  &
  & $\tri\;\cd\BOX[1]=\sd\;\tri\BOX[1]$\\
  \end{tabular}
  \caption{Edgewise subdivision, a cubical analogue, and triangulation}
  \label{fig:sd.cd.tri}
\end{figure}

We introduce homotopy theories for cubical sets and streams in \S\ref{sec:equivalence}.
Intuitively, a \textit{directed homotopy} of stream maps $X\ra Y$ should be a stream map $\I\ra Y^X$ from the unit interval $\I$ equipped with some local preordering to the mapping stream $Y^X$.  
The literature \cite{fgr:ditop, grandis:d} motivates two distinct homotopy relations, corresponding to different choices of stream-theoretic structure on $\I$, on stream maps [Figure \ref{fig:homotopies}].  
The weaker, and more intuitive, of the definitions relates stream maps $X\ra Y$ classically homotopic through stream maps.  
We adopt the stronger \cite{grandis:d} of the definitions.
However, we show that both homotopy relations coincide for the case $X$ compact and $Y$ quadrangulable [Theorem \ref{thm:d}], generalizing a result for $X$ a directed unit interval and $Y$ a directed realization of a \textit{precubical set} \cite[Theorem 5.6]{fajstrup:approx}.  

\begin{figure}[h]
  \label{fig:homotopies}
  \begin{tabular}{ccc}
    \includegraphics[width=20mm,height=20mm]{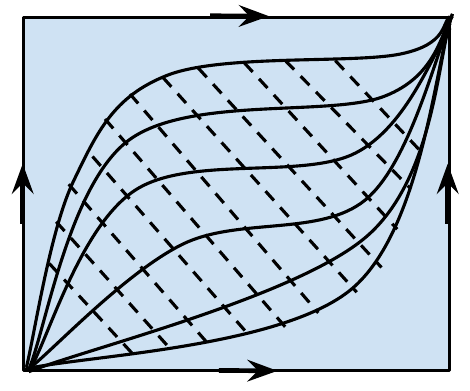} 
  &  
  & \includegraphics[width=20mm,height=20mm]{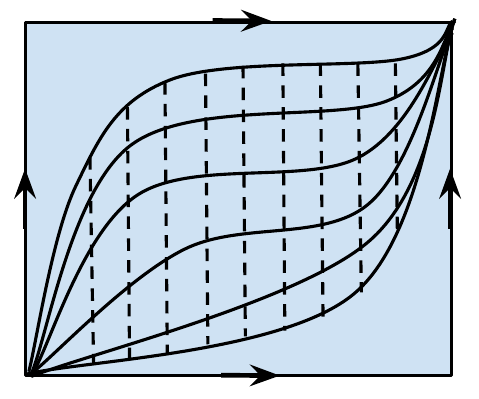}
  \\
  \end{tabular}
  \caption{Weak (left) and strong (right) types of directed homotopies. The homotopies on both sides are identical up to reparametrization of paths.  Only the right homotopy is monotone in the homotopy coordinate, traced by the dotted lines.  For compact quadrangulable streams, the equivalence relations generated by weak and strong definitions of directed homotopy are equivalent [Theorem \ref{thm:d}].}
  \label{fig:homotopy}
\end{figure}

We establish our main results in \S\ref{subsec:main}, stated in the more general setting of diagrams of streams and diagrams of cubical sets.  
We first prove a directed analogue [Theorem \ref{thm:sd.approx}] of classical simplicial approximation up to subdivision and a dual cubical approximation theorem [Corollary \ref{cor:cd.approx}], all for stream maps having compact domain.    
Our proofs, while analogous to classical arguments, require more delicacy because: simplicial sets and cubical sets do not admit approximations as oriented simplicial complexes and cubical complexes, even up to iterated subdivision; and cubical functions are more difficult to construct than simplicial functions.  
In a future sequel, we formalize a sense in which stream realization induces an equivalence between weak homotopy theories of cubical sets and streams.
\section{Conventions}

We fix some categorical notation and terminology.

\subsubsection{General}
We write $\SETS$ for the category of sets and functions.
We write $f_{\restriction X}$ for the restriction of a function $f:Y\ra Z$ to a subset $X\subset Y$.
We similarly write $F_{\restriction\mathscr{A}}:\mathscr{A}\ra\mathscr{C}$ for the restriction of a functor $F:\mathscr{B}\ra\mathscr{C}$ to a subcategory $\mathscr{A}\subset\mathscr{B}$.  
We write $S\cdot c$ for the $S$-indexed coproduct in a category $\GENERIC$ of distinct copies of a $\GENERIC$-object $c$.
Let $\DIAGRAM$ denote a small category.  
For each $\DIAGRAM$, cocomplete category $\GENERIC$, and functor $F:\OP\DIAGRAM\times\DIAGRAM\ra\GENERIC$, we write 
\begin{equation*}
\int_{\DIAGRAM}^{g}F(g,g)
\end{equation*}
for the coend \cite{maclane:categories} of $F$; see Appendix \S\ref{sec:coends} for details.    
For each Cartesian closed category $\GENERIC$, we write $-^c$ for the right adjoint to the functor $-\times c:\GENERIC\ra\GENERIC$, for each $\GENERIC$-object $c$.  

\subsubsection{Presheaves}
Fix a small category $\DIAGRAM$.  
We write $\PRESHEAVES\DIAGRAM$ for the category of presheaves $\OP\DIAGRAM\ra\SETS$ on $\DIAGRAM$ and natural transformations between them, call a $\PRESHEAVES\DIAGRAM$-object $B$ a \textit{subpresheaf} of a $\PRESHEAVES\DIAGRAM$-object $C$ if $B(g)\subset C(g)$ for all $\DIAGRAM$-objects $g$ and $B(\gamma)$ is a restriction and corestriction of $C(\gamma)$ for each $\DIAGRAM$-morphism $\gamma$, and write $\DIAGRAM[-]:\DIAGRAM\ra\PRESHEAVES\DIAGRAM$ for the Yoneda embedding naturally sending a $\DIAGRAM$-object $g$ to the representable presheaf
$$\DIAGRAM[g]=\DIAGRAM(-,g):\OP\DIAGRAM\ra\SETS.$$

For each $\PRESHEAVES\DIAGRAM$-morphism $\psi:C\ra D$ and subpresheaf $B$ of $C$, we write
$$\psi_{\restriction B}:B\ra D$$
for the component-wise restriction of $\psi$ to a $\PRESHEAVES\DIAGRAM$-morphism $B\ra D$.  
When a $\PRESHEAVES\DIAGRAM$-object $B$ and a $\DIAGRAM$-object $g$ are understood, for each $\sigma\in B(g)$ we write $\langle\sigma\rangle$ for the smallest subpresheaf $A\subset B$ satisfying $\sigma\in A(g)$ and $\sigma_*$ for the image of $\sigma$ under the following natural bijection defined by the Yoneda embedding:
\begin{equation*}
  B(g)\cong\PRESHEAVES\DIAGRAM(\DIAGRAM[g],B)
\end{equation*}

\subsubsection{Supports}
We will often talk about the support of a point in a geometric realization and the carrier of a simplex in a simplicial subdivision.
We provide a uniform language for such generalized supports.

Fix a category $\GENERIC$.  
For each $\GENERIC$-object $c$, a \textit{subobject} of $c$ is an equivalence class of monos of the form $s\ra c$, where two monos $s'\ira c$ and $s''\ira c$ are equivalent if each factors the other in $\GENERIC$.   
We often abuse notation and treat a subobject of a $\GENERIC$-object $c$ as a $\GENERIC$-object $s$ equipped with a distinguished mono representing the subobject; we write this mono as $s\ira c$.  
For each morphism $\gamma:a\ra c$ in a given complete category, we write $\gamma(a)$ for the \textit{image} of $c$ under $\gamma$, the minimal subobject $b\subset c$ through which $\gamma$ factors.
For an object $c$ in a given category, we write $b\subset c$ to indicate that $b$ is a subobject of $c$.  

\begin{defn}
  \label{defn:supports}
  Consider a well-powered category $\mathscr{B}$ closed under intersections of subobjects and a functor $F:\mathscr{B}\ra\mathscr{A}$ preserving monos.
  For each $\mathscr{B}$-object $b$ and subobject $a$ of the $\mathscr{A}$-object $Fb$, we write $\supp_F(a,b)$ for
  \begin{equation}
    \supp_F(a,b)=\bigcap\;\{b'\;|\;b'\subset b,\;a\subset Fb'\},
  \end{equation}
  the unique minimal subobject $b''\subset b$ such that $a\subset Fb''$.
\end{defn}

We wish to formalize the observation that supports of small objects are small.  
Fix a category $\GENERIC$.  
A $\GENERIC$-object $c$ is \textit{connected} if $c$ is not initial and the functor $\GENERIC(c,-)$ from $\GENERIC$ to the category of sets and functions preserves coproducts and \textit{projective} if $\GENERIC(c,\epsilon):\GENERIC(c,x)\ra\GENERIC(c,y)$ is surjective for each epi $\epsilon:x\ra y$.
We call a $\GENERIC$-object $c$ \textit{atomic} if $c$ is the codomain of an epi in $\GENERIC$ from a projective connected $\GENERIC$-object.
The atomic objects in the category $\SETS$ of sets and functions are the singletons, the atomic objects in the categories $\PRESHEAVES\DEL$ and $\PRESHEAVES\BOX$ of simplicial sets and cubical sets are those simplicial sets and cubical sets of the form $\langle\sigma\rangle$.
Our motivating examples for connected and projective objects are representable simplicial sets, representable cubical sets, and singleton spaces.

\begin{defn}
  \label{lem:atomicity}
  An object $a$ in a category $\GENERIC$ is \textit{atomic} if there exists an epi
  $$p\ra a$$
  in $\GENERIC$ with $p$ connected and projective.
\end{defn}

We make a general observation about supports with respect to a functor $F$.    
Our motivating examples of $F$ in the following lemma are subdivisions, a \textit{triangulation} operator converting cubical sets into simplicial sets, and functors taking simplicial sets and cubical sets to the underlying sets of their geometric realizations.

\begin{lem}
  \label{lem:atomic.supports}
  Consider a pair of small categories $\DIAGRAM_1$, $\DIAGRAM_2$ and functor 
  $$F:\PRESHEAVES\DIAGRAM_1\ra\PRESHEAVES\DIAGRAM_2$$
  preserving coproducts, epis, monos, and intersections of subobjects.
  For each $\PRESHEAVES\DIAGRAM_1$-object $b$ and atomic $\GENERIC_2$-subobject $a\subset Fb$, $\supp_F(a,b)$ is the image of a representable presheaf.
\end{lem}
\begin{proof}
Let $\iota$ be a monic $a\ira Fb$.
We assume $b=\supp_F(a,b)$ without loss of generality.
Let $\epsilon$ be the natural epi $\coprod_{g}b(g)\cdot\DIAGRAM_1[g]\ra b$.
There exists an epi $\alpha:p\ra a$ with $p$ connected projective by $a$ atomic.
We can make the identification $F\coprod_{g}b(g)\cdot\DIAGRAM_1[g]=\coprod_gb(g)\cdot F\DIAGRAM_1[g]$ because $F$ preserves coproducts.  
Moreover, $F\epsilon$ is epi because $F$ preserves epis.  
Thus there exists a morphism $\hat\iota:p\ra\coprod_{g}b(g)\cdot F\DIAGRAM_1[g]$ such that $(F\epsilon)\hat\iota=\iota\alpha$ by $p$ projective and $F\epsilon$ epi.
There exists a $\DIAGRAM_1$-object $g$ and $\sigma\in b(g)$ such that $\hat\iota(p)\subset F(\sigma\cdot\DIAGRAM_1[g])$ by $p$ connected.
Therefore $a=\iota\alpha(p)\subset(F\epsilon)(\hat\iota p)\subset(F\epsilon)(F(\sigma\cdot\DIAGRAM_1[g]))=F(\epsilon(\sigma\cdot\DIAGRAM_1[g]))$.
Thus $b=\supp_F(a,b)\subset\epsilon(\sigma\cdot\DIAGRAM_1[g])$ by $\supp_F(a,b)$ minimal.  
Conversely, $\epsilon(\sigma\cdot\DIAGRAM_1[g])\subset b$ because $\epsilon$ has codomain $b$.  
Thus $b=\epsilon(\sigma\cdot\DIAGRAM_1[g])$.
\end{proof}

\subsubsection{Order theory}
We review some order-theoretic terminology in Appendix \S\ref{sec:preorders}.
For each preordered set $X$, we write $\leqslant_X$ for the preorder on $X$ and $\graph{\leqslant_X}$ for its \textit{graph}, the subset of $X\times X$ consisting of all pairs $(x,y)$ such that $x\leqslant_Xy$.  
We write $[n]$ for the set $\{0,1,\ldots,n\}$ equipped with its standard total order and $[-1]$ for the empty preordered set. 
For a(n order-theoretic) lattice $L$, we write $\join_L$ and $\meet_L$ for the join and meet operations $L^2\ra L$. 

\begin{eg}
  For all natural numbers $n$ and $i,j\in[n]$,
  $$i\meet_{[n]}j=\min(i,j),\quad i\join_{[n]}j=\max(i,j).$$
\end{eg}

\section{Streams}\label{sec:streams}
Various formalisms \cite{fgr:ditop, grandis:d, krishnan:convenient} model topological spaces equipped with some compatible temporal structure.  
A category $\STREAMS$ of \textit{streams}, spaces equipped with local preorders \cite{krishnan:convenient}, suffices for our purposes due to the following facts: the category $\STREAMS$ is Cartesian closed \cite[{Theorem 5.12}]{krishnan:convenient}, the forgetful functor from $\STREAMS$ to the category $\SPACES$ of compactly generated spaces creates limits and colimits \cite[Proposition 5.8]{krishnan:convenient}, and $\STREAMS$ naturally contains a category of connected compact Hausdorff topological lattices as a full subcategory \cite[Theorem 4.7]{krishnan:convenient}.

\begin{defn}
  \label{defn:streams}
  A \textit{circulation} $\leqslant$ on a space $X$ is a function assigning to each open subset $U$ of $X$ a preorder $\leqslant_U$ on $U$ such that for each collection $\mathcal{O}$ of open subsets of $X$, $\leqslant_{\bigcup\mathcal{O}}$ is the preorder on $\bigcup\mathcal{O}$ with smallest graph containing
\begin{equation}
  \label{eqn:cosheaf}
  \bigcup_{U\in\mathcal{O}}\graph{\leqslant_U}.
\end{equation}
The circulation $\leqslant$ on a weak Hausdorff space $X$ is a \textit{k-circulation} if for each open subset $U\subset X$ and pair $x\leqslant_Uy$, there exist compact Hausdorff subspace $K\subset U$ and circulation $\leqslant'$ on $K$ such that $x\leqslant'_{K}y$ and $\graph{\leqslant'_{K\cap V}}\subset\graph{\leqslant_V}$ for each open subset $V$ of $X$. 
\end{defn}

A circulation is the data of a certain type of ``cosheaf.''
A k-circulation is a cosheaf which is ``compactly generated.''

\begin{defn}
  \label{defn:topological.categories}
  A \textit{stream} $X$ is a weak Hausdorff k-space equipped with a k-circulation on it, which we often write as $\leqslant$.
  A \textit{stream map} is a continuous function 
  $$f:X\ra Y$$
  from a stream $X$ to a stream $Y$ satisfying $f(x)\leqslant_Uf(y)$ whenever $x\leqslant_{f^{-1}U}y$, for each open subset $U$ of $Y$.
  We write $\SPACES$ for the category of weak Hausdorff k-spaces and continuous functions, $\STREAMS$ for the category of streams and stream maps, and $\PREORDEREDSETS$ for the category of preordered sets and monotone functions.
\end{defn}

Just as functions to and from a space induce initial and final topologies, continuous functions to and from streams induce initial and final circulations \cite[Proposition 5.8]{krishnan:convenient} and \cite[Proposition 7.3.8]{borceux:2} in a suitable sense. 

\begin{prop}\label{prop:topological}
  The forgetful functor $\STREAMS\ra\SPACES$ is topological.
\end{prop}

In particular, the forgetful functor $\STREAMS\ra\SPACES$ creates limits and colimits.  

\begin{prop}[{\cite[Lemma 5.5, Proposition 5.11]{krishnan:convenient}}]\label{prop:almost.topological}
  The forgetful functor
  $$\STREAMS\rightarrow\PREORDEREDSETS,$$
  sending each stream $X$ to its underlying set equipped with $\leqslant_X$, preserves colimits and finite products.
\end{prop}

\begin{thm}[{\cite[Theorem 5.13]{krishnan:convenient}}]
  \label{thm:x.closed}
  The category $\STREAMS$ is Cartesian closed.
\end{thm}

An equalizer in $\STREAMS$ of a pair $X\rightrightarrows Y$ of stream maps is a stream map $e:E\ra X$ such that $e$ defines an equalizer in $\SPACES$ and $e$ is a \textit{stream embedding}.

\begin{defn}
  A \textit{stream embedding} $e$ is a stream map $Y\ra Z$ such that for all stream maps $f:X\ra Z$ satisfying $f(X)\subset Y$, there exists a unique dotted stream map making the following diagram commute.
  \begin{equation*}
    \xymatrix{
      X\ar[r]^f\ar@{.>}[d]
    & Z
    \\
      Y\ar[ur]_{e}
    }
  \end{equation*}
\end{defn}

Stream embeddings define topological embeddings of underlying spaces [Proposition \ref{prop:topological}].  
Inclusions, from open subspaces equipped with suitable restrictions of circulations, are embeddings.  
However, general stream embeddings are difficult to explicitly characterize.  
We list some convenient criteria for a map to be an embedding. 
The following criterion follows from the definition of $k$-circulations.  

\begin{lem}
  \label{lem:k-embeddings}
  For a stream map $f:X\ra Y$, the following are equivalent.
  \begin{enumerate}
    \item The map $f$ is a stream embedding.
    \item For each stream embedding $k:K\ra X$ from a compact Hausdorff stream $K$, $fk$ is a stream embedding.
  \end{enumerate}
\end{lem}

The following criterion, straightforward to verify, is analogous to the statement that a sheaf $\mathcal{F}$ on a space $X$ is the pullback of a sheaf $\mathcal{G}$ on a space $Y$ along an inclusion $i:X\ira Y$ if for each open subset $U$ of $X$, $\mathcal{F}_U$ is the colimit, taken over all open subsets $V$ of $Y$ containing $U$, of objects of the form $\mathcal{G}_V$.

\begin{lem}
  \label{lem:restrictions.define.subcirculations}
  A stream map $f:X\ra Y$ is a stream embedding if
  $$\graph{\leqslant^X_U}=U^2\cap\bigcap_{V\in\mathscr{B}_U}\graph{\leqslant^{Y}_V},$$
  for each open subset $U$ of $X$, where $\mathscr{B}_U$ denotes a choice of open neighborhood basis in $Y$ of $U$, $\leqslant^X$ and $\leqslant^Y$ are the respective circulations on $X$ and $Y$, and $f$ defines an inclusion of a subspace.  
\end{lem}

A \textit{topological lattice} is a(n order theoretic) lattice $L$ topologized so that its join and meet operations $\join_L,\meet_L:L^2\ra L$ are jointly continuous.

\begin{defn}
  We write $\vec{\I}$ for the unit interval
  $$\I=[0,1],$$
  regarded as a topological lattice whose join and meet operations are respectively defined by maxima and minima.  
\end{defn}

We can regard a category of connected and compact Hausdorff topological lattices as a full subcategory of $\STREAMS$ \cite[Propositions 4.7, 5.4, 5.11]{krishnan:convenient}, \cite[Proposition 1, Proposition 2, and Theorem 5]{nachbin:order}, \cite[Proposition VI-5.12 (i)]{ghklms:lattices}, \cite[Proposition VI-5.15]{ghklms:lattices}.

\begin{thm}\label{thm:embed}
  There exists a unique dotted embedding making the diagram
  \begin{equation*}\label{eqn:embed}
    \xymatrix{
      \POSPACES\ar[r]\ar[d]\ar@{.>}[dr]
    & \PREORDEREDSETS
    \\
      \SPACES
    & \STREAMS,\ar[l]\ar[u]
    }
  \end{equation*}
  where $\POSPACES$ is the category of compact Hausdorff connected topological lattices and monotone maps between them and the solid arrows are forgetful functors, commute.
  Moreover, the dotted embedding is full.
\end{thm}

We henceforth regard connected compact Hausdorff topological lattices, such as $\vec{\I}$, as streams.

\section{Simplicial models}\label{sec:simplicial}
Simplicial sets serve as technically convenient models of directed spaces.  
Firstly, edgewise (ordinal) subdivision, the subdivision appropriate for preserving the directionality encoded in simplicial orientations, is simple to define [Definition \ref{defn:sd}] and hence straightforward to study [Lemma \ref{lem:sd.fold}].  
Secondly, the graphs of natural preorders $\leqslant_{\direalize{\;B\;}}$ on geometric realizations $|B|$ of simplicial sets $B$ admit concise descriptions in terms of the structure of $B$ itself [Lemma \ref{lem:simplicial.preorders}].  

\subsection{Simplicial sets}\label{subsec:simplicial.sets}
We write $\DEL$ for the category of finite non-empty ordinals
\begin{equation*}
[n]=\{0\leqslant_{[n]}1\leqslant_{[n]}\cdots\leqslant_{[n]}n\},\;\; n=0,1,\ldots
\end{equation*}
and monotone functions between them.
\textit{Simplicial sets} are $\hat\DEL$-objects and \textit{simplicial functions} are $\hat\DEL$-morphisms.
We refer the reader to \cite{may:simplicial} for the theory of simplicial sets. 
The representable simplicial sets $\DEL[n]=\DEL(-,[n])$ model combinatorial simplices.  
The \textit{vertices} of a simplicial set $B$ are the elements of $B([0])$.  
For each simplicial set $B$, we write $\dim B$ for the infimum of all natural numbers $n$ such that the natural simplicial function $B[n]\cdot[n]\ra B$ is epi.
For each atomic simplicial set $A$, there exists a unique $\sigma\in A[\dim A]$ such that $A=\langle\sigma\rangle$.  
Atomic simplicial sets are those simplicial sets of the form $\langle\sigma\rangle$.
Every atomic simplicial set has a ``minimum vertex,'' defined as follows.  

\begin{defn}
  \label{defn:minimum.of.atomic.simplicial.set}
  For each atomic simplicial set $A$, we write $\min A$ for the vertex
  $$\min A=\sigma_{[0]}(0),$$
  where $\sigma\in A[\dim A]$ satisfies $A=\langle\sigma\rangle$.
\end{defn}

Preordered sets naturally define simplicial sets via the following nerve functor.

\begin{defn}
  We write $\sn$ for the functor
  $$\sn:\PREORDEREDSETS\ra\PRESHEAVES\DEL$$
  naturally sending each preordered set $P$ to the simplicial set $\sn P = \PREORDEREDSETS(-,P)_{\restriction\OP\DEL}$.  
  We identify $(\sn P)[0]$ with $P$ for each preordered set $P$.  
\end{defn}

Simplicial sets naturally admit ``simplicial preorders'' as follows.  
The simplicial sets $\DEL[n]$ naturally lift along the forgetful functor $\PREORDEREDSETS\ra\SETS$ to functors
$$([n]^{-})_{\restriction\OP\DEL}:\OP\DEL\ra\PREORDEREDSETS$$
naturally sending each ordinal $[m]$ to the preordered set $[n]^{[m]}$ of monotone functions $[m]\ra[n]$ equipped with the point-wise partial order.  
General simplicial sets naturally lift to $\PREORDEREDSETS$ as follows.

\begin{defn}
  \label{defn:preordered.presheaves}
  We sometimes identify a simplicial set $B$ with the functor
  \begin{equation*}
    \int_{\DEL}^{[n]}B([n])\cdot([n]^{-})_{\restriction\OP\DEL}:\OP\DEL\ra\PREORDEREDSETS.
  \end{equation*}
  We write $\graph{\leqslant_B}$ for the subpresheaf of $B^2$ naturally assigning to each non-empty finite ordinal $[n]$ the graph $\graph{\leqslant_{B[n]}}\subset(B[n])^2$ of the preorder on the preordered set $B[n]$.
\end{defn}

\begin{lem}
  \label{lem:graph.nerves}
  There exists a bijection
  \begin{equation}
    \label{eqn:graph.nerves}
    \sn P^{[1]}\cong\graph{\leqslant_{\sn P}}
  \end{equation}
  natural in preordered sets $P$.  
\end{lem}
\begin{proof}
  There exist isomorphisms
  \begin{align*}
    (\sn P^{[1]})[n] 
    \cong\;& \PREORDEREDSETS([n],P^{[1]}) && \text{definition of $\sn$}\\
    \cong\;& \PREORDEREDSETS([1],P^{[n]}) && \text{$\PREORDEREDSETS$ Cartesian closed}\\
    \cong\;& \graph{\leqslant_{P^{[n]}}} && \graph{\leqslant_X}=X^{[1]} \\
    \cong\;& \graph{\leqslant_{\sn P}}[n] & 
  \end{align*}
  of sets natural in preordered sets $P$ and $\DEL$-objects $[n]$.  
  The last line holds because $P^{[n]}$ is the colimit of all preordered sets $[k]^{[n]}$, taken over all monotone functions $[k]\ra P$.  
\end{proof}

\subsection{Subdivisions}\label{subsec:sd}
Edgewise subdivision \cite{segal:edgewise}, otherwise known as ordinal subdivision \cite{ep:ordinal}, plays a role in directed topology analogous to the role barycentric subdivision \cite{curtis:approx} plays in classical topology.
A description \cite{ep:ordinal} of edgewise subdivision in terms of \textit{ordinal sums} in $\DEL$ makes it convenient for us to reason about double edgewise subdivision [Lemma \ref{lem:sd.fold}].  
We refer the reader to \cite{ep:ordinal} for more details and Figure \ref{fig:sd} for a comparison between edgewise and barycentric subdivisions.

\begin{defn}
  We write $\oplus$ for the tensor, sending pairs $[m],[n]$ of finite ordinals to $[m+n+1]$ and pairs $\phi':[m']\ra[n']$, $\phi'':[m'']\ra[n'']$ of monotone functions to the monotone function $(\phi'\oplus\phi''):[m'+m''+1]\ra[n'+n''+1]$ defined by
  \begin{equation*}
  (\phi'\oplus\phi'')(k)=
  \begin{cases}
    \phi'(k), & k=0,1,\ldots,m'\\
    n'+1+\phi''(k-m'-1), & k=m'+1,m'+2,\ldots,m'+m''+1\\
  \end{cases},
  \end{equation*}
  on the category of finite ordinals $[-1]=\varnothing$, $[0]=\{0\}$, $[1]=\{0<1\}$, $[2]=\{0<1<2\}$, \ldots and monotone functions between them.  
\end{defn}

In particular, the empty set $[-1]=\varnothing$ is the unit of the tensor $\oplus$.  
We can thus define natural monotone functions $[n]\ra[n]\oplus[n]$ as follows.  

\begin{defn}
  We write $\gamma^\DEL_{[n]}$, $\bar\gamma^\DEL_{[n]}$ for the monotone functions
  $$\gamma^{\DEL}_{[n]}=\id_{[n]}\oplus([-1]\ra[n]),\bar\gamma^{\DEL}_{[n]}=([-1]\ra[n])\oplus\id_{[n]}:[n]\ra[n]\oplus[n],$$
  natural in finite ordinals $[n]$.
\end{defn}
 
In other words, $\gamma^{\DEL}_{[n]}$ and $\bar\gamma^{\DEL}_{[n]}$ are the functions $[n]\ra[2n+1]$ defined by
$$\gamma^{\DEL}_{[n]}(i)=i,\quad\bar\gamma^{\DEL}_{[n]}(i)=i+n+1.$$

\begin{defn}
  \label{defn:sd}
  We write $\sd$ for the functor $\PRESHEAVES\DEL\ra\PRESHEAVES\DEL$ induced from the functor
  $$(-)^{\oplus 2}:\DEL\ra\DEL.$$
\end{defn}

\begin{figure}[ht]
  \begin{tabular}{cccc}
    \includegraphics[width=20mm,height=20mm]{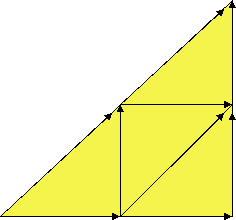} 
  & \includegraphics[width=20mm,height=20mm]{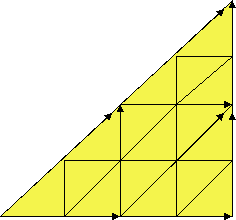}
  & \includegraphics[width=20mm,height=20mm]{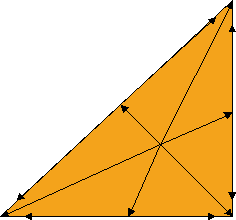} 
  & \includegraphics[width=20mm,height=20mm]{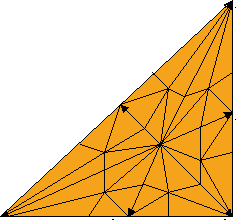}
  \\
    edgewise & double edgewise & barycentric & double barycentric \\
  \end{tabular}
  \caption{Simplicial subdivisions of $\DEL[2]$}
  \label{fig:sd}
\end{figure}

Natural monotone functions $[n]\ra[n]\oplus[n]$ induce natural simplicial functions $\sd B\ra B$, defined as follows.  

\begin{defn}
  We write $\gamma,\bar\gamma$ for the natural transformations
  $$\gamma,\bar\gamma:\sd\ra\id_{\PRESHEAVES\DEL}:\PRESHEAVES\DEL\ra\PRESHEAVES\DEL$$
  respectively induced from the natural monotone functions $\gamma^\DEL_{[n]},\bar\gamma^\DEL_{[n]}:[n]\ra[n]\oplus[n]$.  
\end{defn}

The functor $\sd$, and hence $\sd^2$, are left and right adjoints \cite[\S4]{ep:ordinal} and hence $\sd^2$ preserves monos, intersections of subobjects, and colimits.
Thus we can construct $\supp_{\sd^2}(B,C)\subset C$ [Definition \ref{defn:supports}] for simplicial sets $C$ and subpresheaves $B\subset\sd^2C$.  
The following observations about double edgewise subdivisions of the standard $1$-simplex $\DEL[1]$ later adapt to the cubical setting [Lemmas \ref{lem:cd.fold} and \ref{lem:cd.star.collapse}].  

\begin{lem}
  \label{lem:sd.star.collapse}
  For all atomic subpresheaves $A\subset\sd^2\DEL[1]$ and $v\in A[0]$,
  \begin{equation*}
    (\gamma\bar\gamma)_{\DEL[1]}(A)\subset\supp_{\sd^2}(\langle v\rangle,\DEL[1]).
  \end{equation*}
\end{lem}
\begin{proof}
  Each $v\in A[0]\subset\DEL([0]\oplus[0]\oplus[0]\oplus[0],[1])$ is a monotone function
  $$v:[3]\ra[1]$$
  
  The case $v:[3]\ra[1]$ non-constant holds because
  \begin{equation*}
    \supp_{\sd^2}(\langle v\rangle,\DEL[1])=\DEL[1].
  \end{equation*}
  
  Consider the case $v$ a constant function. 
  Let $n=\dim A$ and $\sigma$ be the monotone function $[n]\oplus[n]\oplus[n]\oplus[n]\ra[1]$, an element in $(\sd^2\DEL[1])[n]$, such that $A=\langle\sigma\rangle$.
  There exists a $k\in[n]$ such that $v(0)=\sigma(k)$, $v(1)=\sigma(k+n+1)$, $v(2)=\sigma(k+2n+2)$ and $v(3)=\sigma(k+3n+3)$ because $v\in\langle\sigma\rangle[0]$.  
  Therefore
  \begin{equation*}
    v(0)=\sigma(k)\leqslant_{[1]}\sigma(k+1)\leqslant_{[1]}\cdots\leqslant_{[1]}\sigma(k+2n+2)=v(0),
  \end{equation*}
  hence $\sigma(n+1+i)=v(0)$ for all $i\in[n]$, hence
  \begin{eqnarray*}
        (((\gamma\bar\gamma)_{\DEL[1]})_{[n]}\sigma)(i)
    &=& (\sigma(\bar\gamma^{\DEL}_{[n]}\oplus\bar\gamma^{\DEL}_{[n]})(\gamma^{\DEL}_{[n]})(i))\\
    &=& \sigma(n+1+i)\\
    &=& v(0)
  \end{eqnarray*}
  for all $i=0,1,\ldots,n$.  
  Thus $((\gamma\bar\gamma)_{\DEL[1]})_{[n]}\sigma$ is a constant function $[n]\ra[1]$ taking the value $v(0)$ and hence $(\gamma\bar\gamma)_{\DEL[1]}A=\langle v(0)\rangle=\supp_{\sd^2}(\langle v\rangle,\sd^2\DEL[1])$. 
\end{proof}

\begin{lem}
  \label{lem:sd.fold}
  For each atomic subpresheaf $A\subset\sd^2\DEL[1]$, there exists a unique minimal subpresheaf $B\subset\DEL[1]$ such that $A\;\cap\;\sd^2B\neq\varnothing$.
  Moreover, the diagram
  \begin{equation}
  \label{eqn:sd.fold}
  \xymatrix@C=4pc{
    **[l]A\ar[r]\ar[d]_{\pi}
    & \sd^2\DEL[1]\ar[r]^{(\gamma\bar\gamma)_{\DEL[1]}}
    & \DEL[1]\ar@{=}[d]\\
    **[l]A\cap\sd^2B\ar[r]
    & \sd^2\DEL[1]\ar[r]_{(\gamma\bar\gamma)_{\DEL[1]}}
    & \DEL[1],
    }
  \end{equation}
  where $\pi:A\ra A\cap\sd^2B$ is the unique retraction of $A$ onto $A\cap\sd^2B$ and the unlabelled solid arrows are inclusions, commutes.  
\end{lem}
\begin{proof}
  In the case $A\cap\sd^2\langle 0\rangle=A\cap\sd^2\langle 1\rangle=\varnothing$, then $\DEL[1]$ is the unique choice of subpresheaf $B\subset\DEL[1]$ such that $A\cap\sd^2B\neq\varnothing$ and hence $\id_{\DEL[1]}$ is the unique choice of retraction $\pi$ making (\ref{eqn:sd.fold}) commute.

  It therefore remains to consider the case $A\cap\sd^2\langle 0\rangle\neq\varnothing$ and the case $A\cap\sd^2\langle 1\rangle\neq\varnothing$ because the only non-empty proper subpresheaves of $\DEL[1]$ are $\langle 0\rangle$ and $\langle 1\rangle$.  
  We consider the case $A\cap\sd^2\langle 0\rangle\neq\varnothing$, the other case following similarly.  
  Observe
  \begin{equation}
    \label{eqn:sd.fold.proof}
    (\gamma\bar\gamma)_{\DEL[1]}A\subset\langle 0\rangle
  \end{equation}
  [Lemma \ref{lem:sd.star.collapse}.]  
  In particular, $A\cap\sd^2\langle 1\rangle=\varnothing$ because $(\gamma\bar\gamma)_{\langle 1\rangle}(\sd^2\langle 1\rangle)\subset\langle 1\rangle$ [Lemma \ref{lem:sd.star.collapse}].
  Thus $\langle 0\rangle$ is the minimal subpresheaf $B'$ of $B$ such that $A\cap\sd^2 B'\neq\varnothing$.  
  Moreover, the terminal simplicial function $\pi:\DEL[1]\ra\DEL[0]$ makes (\ref{eqn:sd.fold}) commute by (\ref{eqn:sd.fold.proof}).
\end{proof}

\subsection{Realizations}\label{subsec:simplicial.realizations}
The topological $n$-simplex is the subspace
$$\{(t_0,\ldots,t_n)\in[0,1]^{n+1}\;|\;\sum_{i=0}^nt_i=1\}$$
of $\R^{n+1}$.
For convenience, we denote a point $(t_0,\cdots,t_n)$ in the topological $n$-simplex as a formal linear sum $\sum_{i=0}^n t_i(i)$ of abstract symbols $(0),(1),\cdots,(n)$.
Classical geometric realization is the unique functor
$$|-|:\PRESHEAVES\DEL\ra\SPACES$$
preserving colimits, assigning to each simplicial set $\DEL[n]$ the topological $n$-simplex, and assigning to each simplicial function of the form $\DEL[\phi:[m]\ra[n]]:\DEL[m]\ra\DEL[n]$ the linear map $|\DEL[m]|\ra|\DEL[n]|$ defined by
$$|\DEL[\phi]|(\sum_{i=0}^mt_i(i))=\sum_{i=0}^mt_i(\phi(i)).$$
For each simplicial set $B$ and $v\in B[0]$, we write $|v|$ for the image of $|v_*|:|\DEL[0]|\ra|B|$ and call $|v|$ a \textit{geometric vertex} in $|B|$.  

\begin{defn}
  \label{defn:simplicial.direalizations}
  We write $\direalize{-}$ for the unique functor
  $$\direalize{-}\;:\PRESHEAVES\DEL\ra\STREAMS$$
  naturally preserving colimits, assigning to each simplicial set $\DEL[n]$ the space $|\DEL[n]|$ equipped with respective lattice join and meet operations
  $$|\sn\join_{[n]}|,|\sn\meet_{[n]}|:|\sn[n]|^2\ra|\sn[n]|,$$
  and assigning to each simplicial function $\psi:B\ra C$ the stream map $\direalize{B}\ra\direalize{C}$ defined by $|\psi|:|B|\ra|C|$.  
\end{defn}

\begin{figure}[h]
  \begin{tabular}{c}
    \includegraphics[width=30mm,height=30mm]{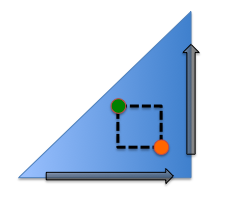} 
  \end{tabular}
  \caption{Lattice structure on a topological simplex.  The corners of the simplex are totally ordered, starting from the bottom left all the way to the top right. The dotted lines connect two generic points to their infimum (bottom left of dotted square) and supremum (top right of dotted square).}
  \label{fig:hep.failure}
\end{figure}

We can directly relate the circulation of a stream realization $\direalize{B}$ with the simplicial structure of $B$ as follows.  

\begin{lem}
  \label{lem:simplicial.preorders}
  There exists a bijection of underlying sets
  $$\graph{\leqslant_{\direalize{\;B\;}}}\cong|\graph{\leqslant_{B}}|$$
  natural in simplicial sets $B$.  
\end{lem}
\begin{proof}
  For the case $B=\DEL[n]$, 
  \begin{align*}
    &\;\;\;\;\graph{\leqslant_{\direalize{\;\DEL[n]\;}}}\\
    &\cong \lim(|\sn\meet_{[n]}|\times|\sn\join_{[n]}|,\id_{|\sn[n]|^2}:|\sn[n]|^2\ra|\sn[n]|^2)\\
    &\cong |\sn\lim(\meet_{[n]}\times\join_{[n]},\id_{[n]^2}:[n]^2\ra[n]^2)|\\
    &\cong |\sn [n]^{[1]}|\\
    &\cong |\graph{\leqslant_{\DEL[n]}}|.
  \end{align*}
  Among the above bijections, the first and third follow from Lemma \ref{lem:lattices}, the second follows from $|\sn-|$ finitely continuous, and the last follows from Lemma \ref{lem:graph.nerves}.
  The general case follows because the forgetful functor $\STREAMS\ra\PREORDEREDSETS$ preserves colimits [Proposition \ref{prop:almost.topological}].
\end{proof}

\begin{thm}
  \label{thm:direalization.preserves.finite.products}
  The functor $\direalize{-}\;:\hat\DEL\ra\STREAMS$ preserves finite products.
\end{thm}
\begin{proof}
  There exist bijections
  \begin{align*}
    &\;\;\;\;\graph{\leqslant_{\direalize{\;\sn L\;}}}\\
    &\cong |\graph{\leqslant_{\sn L}}|\\
    &\cong |\sn L^{[1]}|\\
    &\cong |\sn\lim(\meet_{L}\times\join_{L},\id_{L^2}:L^2\ra L^2)|\\
    &\cong \lim(|\sn\meet_{L}|\times|\sn\join_{L}|,\id_{|\sn L|^2}:|\sn L|^2\ra|\sn L|^2)
  \end{align*}
  natural in lattices $L$, and hence the underlying preordered set of $\direalize{\sn L}$ is a lattice natural in lattices $L$.
  Among the above bijections, the first follows from Lemma \ref{lem:simplicial.preorders}, the second follows from Lemma \ref{lem:graph.nerves}, the third follows from Lemma \ref{lem:lattices}, and the last follows from $|\sn-|$ finitely continuous.

  The universal stream map $\direalize{A\times B}\;\cong\;\direalize{A}\times\direalize{B}$, a homeomorphism of underlying spaces because $|-|$ preserves finite products, thus defines a bijective lattice homomorphism, and hence isomorphism, of underlying lattices for the case $A=\DEL[m]$, $B=\DEL[n]$, hence a stream isomorphism for the case $A=\DEL[m]$, $B=\DEL[n]$ [Theorem \ref{thm:embed}], and hence a stream isomorphism for the general case because finite products commute with colimits in $\hat\DEL$ and $\STREAMS$ [Theorem \ref{thm:x.closed}].   
\end{proof}

We can recover some information about the orientations of a simplicial set $B$ from relations of the form $x\leqslant_{\direalize{\;B\;}}y$ as follows.
Recall our definition [Definition \ref{defn:minimum.of.atomic.simplicial.set}] of the minimum $\min A$ of an atomic simplicial set $A$.  
\begin{lem}
\label{lem:recover.orientations}
For all preordered sets $P$ and pairs $x\leqslant_{\direalize{\;\sn P\;}}y$,
\begin{equation*}
   \label{eqn:recover.orientations}
   \min\supp\;\!\!_{|-|}(\{x\},\sn P)\leqslant_{P}\min\supp\;\!\!_{|-|}(\{y\},\sn P).
\end{equation*}
\end{lem}
\begin{proof}
  The underlying preordered set of $\direalize{\sn P}$ is the colimit, over all monotone functions $[k]\ra P$, of underlying preordered sets of topological lattices $\direalize{\DEL[k]}$ because $\direalize{-}:\PRESHEAVES\DEL\ra\STREAMS$ and the forgetful functor $\STREAMS\ra\PREORDEREDSETS$ are cocontinuous [Proposition \ref{prop:almost.topological}].  

  Therefore it suffices to consider the case $P=[n]$.
  For each $t\in|\DEL[n]|$, let $t_0,\cdots,t_n$ denote the real numbers such that $t=\sum_{i=0}^nt_i(i)$.
  Then
  \begin{align*}
   \sum_{i=0}^ny_i(i)
  &= y && \quad\text{definition of $t_i$'s}\\
  &= x\join_{\direalize{\;\DEL[n]\;}}y && \quad x\leqslant_{\direalize{\;\sn P\;}}y\\
  &= \left(\sum_{i=0}^nx_i(i)\right)\join_{\direalize{\;\DEL[n]\;}}\left(\sum_{j=0}^ny_j(j)\right) && \quad\text{definition of $t_i$'s}\\
  &= \sum_{i,j\in[n]}(x_iy_j)(\max(i,j)) && \quad\text{linearity of $\join_{\direalize{\;\DEL[n]\;}}$}.
  \end{align*}
  We conclude that for each $j=0,1,...,n$,
  $$y_j=\!\!\!\!\!\!\sum_{\max(i,j)=j}\!\!\!\!x_iy_j=\sum_{i=0}^jx_iy_j$$
  and hence $y_j\neq 0$ implies the existence of some $i=0,1,...,j$ such that $x_i\neq 0$.  
  Thus
\begin{eqnarray*}
  \min\direalize{\supp\;\!\!_{|-|}(\{x\},\sn P)}
  &=& \min\{i\;|\;x_i\neq 0\}\\
  &\leqslant_{[n]}& \min\{j\;|\;y_j\neq 0\}\\
  &=&  \min\direalize{\supp\;\!\!_{|-|}(\{y\},\sn P)}.
\end{eqnarray*}
\end{proof}

Edgewise subdivisions respect geometric realizations as follows.

\begin{defn}
  We write $\shomeo_{\DEL[n]}$ for the piecewise linear map
  \begin{equation*}
    \shomeo_{\DEL[n]}:|\sd \DEL[n]|\cong|\DEL[n]|=\nabla[n],
  \end{equation*}
  natural in non-empty finite ordinals $[n]$, characterized by the rule
  \begin{equation*}
    |\phi|\mapsto\half|\phi(0)|+\half|\phi(1)|,\quad\phi\in(\sd\DEL[n])[0]=\DEL([0]\oplus[0],[n]).
  \end{equation*}
\end{defn}

These maps $\shomeo_{\DEL[n]}$, \textit{prism decompositions} in the parlance of \cite{mardisec:prisms}, define homeomorphisms \cite{mardisec:prisms}.  
The restriction of the function $\shomeo_{\DEL[n]}:|\sd\DEL[n]|\ra|\DEL[n]|$ to the geometric vertices $\PREORDEREDSETS([0]\oplus[0],[n])$ of $|\sd\DEL[n]|$ is a lattice homomorphism
$$[n]^{[0]\oplus[0]}\ra\direalize{\DEL[n]}$$
natural in $\DEL$-objects $[n]$ by $\join_{\direalize{\;\DEL[n]\;}},\meet_{\direalize{\;\DEL[n]\;}}:|\DEL[n|^2\ra|\DEL[n]|$ linear.
Thus
$$\shomeo_{\DEL[n]}:\direalize{\DEL[n]}\ra\direalize{\DEL[n]},$$
the linear extension of a lattice homomorphism, defines a lattice isomorphism by $\meet_{\direalize{\;\DEL[n]\;}}$ and $\join_{\direalize{\;\DEL[n]\;}}$ linear, and hence stream isomorphism [Theorem \ref{thm:embed}], natural in ordinals $[n]$.  
We can thus make the following definition.  

\begin{defn}
  We abuse notation and write $\shomeo_B$ for the isomorphism
  \begin{equation*}
    \shomeo_B:\;\direalize{\sd B}\;\cong\;\direalize{B},
  \end{equation*}
  of streams natural in simplicial sets $B$, defined by prism decompositions for the case $B$ representable.  
\end{defn}

\section{Cubical models}\label{sec:cubical}
Cubical sets are combinatorial and economical models of directed spaces \cite{fgr:ditop, goubault:transition}.  
Cubical subdivisions appropriate for directed topology, unlike edgewise subdivisions of simplicial sets, mimic properties of simplicial barycentric subdivision [Lemmas \ref{lem:cd.fold} and \ref{lem:cd.star.collapse}] that make classical simplicial approximation techniques adaptable to the directed setting.  

\subsection{Cubical sets}\label{subsec:cubical.sets}
We refer the reader to \cite{gl:sites} for basic cubical theory.
Definitions of the \textit{box category}, over which cubical sets are defined as presheaves, are not standard \cite{gl:sites}.  
We adopt the simplest of the definitions.

\begin{defn}
  We write $\BOX_1$ for the full subcategory of $\PREORDEREDSETS$ containing the ordinals $[0]$ and $[1]$.  
  Let $\BOX$ be the smallest subcategory of $\PREORDEREDSETS$ containing $\BOX_1$ and closed under binary $\PREORDEREDSETS$-products.  
  We write $\otimes$ for the tensor on $\BOX$ defined by binary $\PREORDEREDSETS$-products.  
\end{defn}

To avoid confusion between tensor and Cartesian products in $\BOX$, we write
$$[0],[1],\boxobj{2},\boxobj{3},\ldots.$$
for the $\BOX$-objects.  
We will often use the following characterization [\cite[Theorem 4.2]{gl:sites} and \cite[Proposition 8.4.6]{cisinski:presheaves}] of $\BOX$ as the free monoidal category generated by $\BOX_1$, in the following sense, without comment.
For each monoidal category $\GENERIC$ and functor $F:\BOX_1\ra\GENERIC$ of underlying categories sending $[0]$ to the unit of $\GENERIC$, there exists a unique dotted monoidal functor, up to natural isomorphism, making the following diagram, in which the vertical arrow is inclusion, commute.  
\begin{equation*}
\xymatrix{
  **[l]\BOX_1\ar[r]^F\ar[d]  
& **[r]\GENERIC\\
  **[l]\BOX,\ar@{.>}[ur]
}
\end{equation*}

Injective $\BOX$-morphisms are uniquely determined by where they send extrema.  
A proof of the following lemma follows from a straightforward verification for $\TRUNCATED{1}{\BOX}$-morphisms and induction.  

\begin{lem}\label{lem:cubical.extrema}
  There exists a unique injective $\BOX$-morphism of the form
  $$\boxobj{m}\ra\boxobj{n}$$
  sending $(0,\cdots,0)$ to $\varepsilon'$ and $(1,\cdots,1)$ to $\varepsilon''$, for each $n=0,1,\ldots$ and $\varepsilon'\leqslant_{\boxobj{n}}\varepsilon''$.
\end{lem}

We regard $\hat\BOX$ as a monoidal category whose tensor $\otimes$ cocontinuously extends the tensor on $\BOX$ along the Yoneda embedding $\BOX[-]:\BOX\ira\PRESHEAVES\BOX$.  
We can regard each tensor product $B\otimes C$ as a subpresheaf of the Cartesian product $B\times C$.  
\textit{Cubical sets} are $\PRESHEAVES\BOX$-objects and \textit{cubical functions} are $\PRESHEAVES\BOX$-morphisms.  
We write the cubical set represented by the box object $\boxobj{n}$ as $\BOX\boxobj{n}$.  
The \textit{vertices} of a cubical set $B$ are the elements of $B([0])$.  
We will sometimes abuse notation and identify a vertex $v$ of a cubical set $B$ with the subpresheaf $\langle v\rangle$ of $B$.  
For each cubical set $B$, we write $\dim B$ for the infimum of all natural numbers $n$ such that the natural cubical function $B[n]\cdot\BOX\boxobj{n}\ra B$ is epi.
For each atomic cubical set $A$, there exists a unique $\sigma\in A\boxobj{\dim A}$ such that $A=\langle\sigma\rangle$.  
Atomic cubical sets, analogous to cells in a CW complex, admit combinatorial boundaries defined as follows.  

\begin{defn}
  \label{defn:cubical.dimension}
  For each atomic cubical set $A$, we write
  $$\partial A$$
  for the unique maximal proper subpresheaf of $A$.  
\end{defn}

Generalizing the quotient $Y/X$ of a set $Y$ by a subset $X$ of $Y$, we write $C/B$ for the object-wise quotient of a cubical set $C$ by a subpresheaf $B$ of $C$.  
For each atomic cubical set $A$, the unique epi $\BOX\boxobj{\dim A}\ra A$ passes to an isomorphism
$$\BOX\boxobj{\dim A}/\partial\BOX\boxobj{\dim A}\cong A/\partial A.$$

Each subpresheaf $B$ of a cubical set $C$ admits a combinatorial analogue of a closed neighborhood in $C$, defined as follows.  

\begin{defn}
  For each cubical set $C$ and subpresheaf $B\subset C$, we write 
  $$\Star_CB$$
  for the union of all atomic subpresheaves of $C$ intersecting $B$.
\end{defn}
  
There exist cubical analogues, defined as follows, of simplicial nerves.

\begin{defn}
  We write $\cn$ for the functor
  $$\cn:\PREORDEREDSETS\ra\PRESHEAVES\BOX$$
  naturally sending each preordered set $P$ to the cubical set $\cn P = \PREORDEREDSETS(-,P)_{\restriction\OP\BOX}$.  
\end{defn}

\subsection{Subdivisions}\label{subsec:cd}
We construct cubical analogues of edgewise subdivisions.  
To start, we extend the category $\BOX$ of abstract hypercubes to a category $\XBOX$ that also models abstract subdivided hypercubes.

\begin{defn}
  We write $\XBOX$ for the smallest sub-monoidal category of the Cartesian monoidal category $\PREORDEREDSETS$ generated by all monotone functions of the form $[0]\ra[1]$, $[1]\ra[0]$, $[2]\ra[1]$, and $[1]\ra[2]$ except the function $[1]\ra[2]$ sending $i$ to $2i$.  
  We write $\otimes$ for the tensor on $\XBOX$.
  We abuse notation and also write $\BOX[-]$ for the functor
  $$\XBOX\ra\PRESHEAVES\BOX$$
  naturally sending each $\XBOX$-object $L$ to the cubical set $\XBOX(-,L)_{\restriction\OP\BOX}:\OP\BOX\ra\SETS$.
\end{defn}

The function $[1]\ra[2]$ defined by doubling is omitted from the definition $\XBOX$ because such a map would define superfluous edges spanning twice the length of the edges in the subdivided cubes illustrated in Figure \ref{fig:sd.cd.tri}.
Context will make clear whether $\BOX[-]$ refers to the Yoneda embedding $\BOX\ra\PRESHEAVES\BOX$ or its extension $\XBOX\ra\PRESHEAVES\BOX$.  
A functor $\double:\BOX\ra\XBOX$ describes the subdivision of an abstract cube as follows.

\begin{defn}
  We write $\double$ for the unique monoidal functor
  $$\double:\BOX\ra\XBOX$$
  sending each $\BOX_1$-object $[n]$ to $[2n]$ and each $\BOX_1$-morphism $\delta:[0]\ra[1]$ to the $\XBOX$-morphism $[0]\ra[2]$ sending $0$ to $2\delta(0)$.  
\end{defn}

Natural $\XBOX$-morphisms $\xboxobj{2}{n}\ra\boxobj{n}$ model cubical functions from subdivided hypercubes to ordinary hypercubes.  

\begin{defn}
  We write $\gamma^{\BOX}_{[1]},\bar\gamma^{\BOX}_{[1]}$ for the monotone functions
  $$\gamma^{\BOX}_{[1]}=\max(1,-)-1,\bar\gamma^{\BOX}_{[1]}=\min(-,1):[2]\ra[1].$$
  More generally, we write $\gamma^{\BOX},\bar\gamma^{\BOX}$ for the unique monoidal natural transformations $\double\ra(\id_{\XBOX})_{\restriction\BOX}$ having the above $[1]$-components.
\end{defn}

We extend our subdivision operation from abstract hypercubes to more general cubical sets.  

\begin{defn}
  \label{defn:cd}
  We write $\cd$ for the unique cocontinuous monoidal functor
  $$\cd:\PRESHEAVES\BOX\ra\PRESHEAVES\BOX$$
  extending $\BOX[-]\circ\double:\BOX\ra\PRESHEAVES\BOX$ along $\BOX[-]:\BOX\ra\PRESHEAVES\BOX$.
\end{defn}

Context will make clear whether $\sd$ refers to simplicial or cubical subdivision.  

\begin{defn}
  We write $\gamma$, $\bar\gamma$ for the natural transformations
  $$\cd\ra\id_{\PRESHEAVES\BOX}:\PRESHEAVES\BOX\ra\PRESHEAVES\BOX$$
  induced from the respective natural transformations $\gamma^{\BOX},\bar\gamma^{\BOX}:\double\ra(\id_{\XBOX})_{\restriction\BOX}$.  
\end{defn}

Context will also make clear whether $\gamma,\bar\gamma$ are referring to the natural simplicial functions $\sd B\ra B$ or natural cubical functions $\sd C\ra C$. 
Cubical subdivision shares some convenient properties with simplicial barycentric subdivision.
The following is a cubical analogue of Lemma \ref{lem:sd.fold}.

\begin{lem}
  \label{lem:cd.fold}
  Fix a cubical set $C$. 
  For each atomic subpresheaf $A\subset\cd^2C$, there exist unique minimal subobject $B\subset C$ such that $A\;\cap\;\cd^2B\neq\varnothing$ and unique retraction $\pi:A\ra A\cap\cd^2B$.  
  The diagram
\begin{equation}
  \label{eqn:cd.fold}
  \xymatrix@C=4pc{
  **[l]A\ar[r]\ar[d]_{\pi}
  & \cd^2C\ar[r]^{(\gamma\bar\gamma)_{C}}
  & C\ar@{=}[d]\\
  **[l]A\cap\cd^2B\ar[r]
  & \cd^2C\ar[r]_{(\gamma\bar\gamma)_{C}}
  & C,
  }
\end{equation}
  whose unlabelled solid arrows are inclusions, commutes.  
  Moreover, $A\cap\cd^2B$ is isomorphic to a representable cubical set.  
\end{lem}

We postpone a proof until \S\ref{subsubsec:proofs}. 
The retractions in the lemma are natural in the following sense.

\begin{lem}
  \label{lem:natural.cd.folding}
  Consider the commutative diagram in $\PRESHEAVES\BOX$ on the left side of
  \begin{equation*}
    \xymatrix{
    A'\ar[r]^{\alpha}\ar[d]
    & A''\ar[d]
    & A'\ar[r]^{\alpha}\ar[d]_{\pi'}
    & A''\ar[d]^{\pi''}
    \\
    \sd^2C'\ar[r]_{\sd^2\psi}
    & \sd^2C''
    & A'\cap\sd^2B'\ar@{.>}[r]
    & A''\cap\sd^2B'',
    }
  \end{equation*}
  where $A'$, $A''$ are atomic and the vertical arrows in the left square are inclusions.  
  Suppose there exist minimal subpresheaves $B'$ of $C'$ such that $A'\cap\sd^2B'\neq\varnothing$ and $B''$ of $C''$ such that $A''\cap\sd^2B''\neq\varnothing$ and there exists retractions $\pi'$ and $\pi''$ of the form above.
  Then there exists a unique dotted cubical function making the right square commute.  
\end{lem}
\begin{proof}
  Uniqueness follows from the retractions epi.  
  It therefore suffices to show existence.  

  Consider the case $\alpha$, $\psi$, and hence also $\sd^2\psi$ epi.
  For each subpresheaf $D''$ of $C''$ such that $A''\cap\cd^2D''\neq\varnothing$, $A'\cap\cd^2\psi^{-1}D''\neq\varnothing$, hence $B'$ is a subpresheaf of $\psi^{-1}(D'')$, hence $\psi$ restricts and corestricts to a cubical function $B'\ra B''$, hence $\cd^2\psi$ restricts and corestricts to both $\alpha:A'\ra A''$ and $\cd^2B'\ra\cd^2B''$, and hence $\cd^2\psi$ restricts and corestricts to our desired cubical function.
  
  Consider the case $\alpha$, $\psi$, and hence also $\cd^2\psi$ monic.
  Then $B''$ is a subpresheaf of $B'$, hence a restriction of the retraction $\pi''$ defines a retraction $A'\cap\cd^2B'\ra A''\cap\cd^2B''$ onto its image making the right square commute because $\alpha$, $\pi''$, and hence $\pi''\alpha$ are retractions onto their images and retractions of atomic cubical sets are unique.
  
  The general case follows because every cubical function naturally factors as the composite of an epi followed by a monic.
\end{proof}

The following is a cubical analogue of Lemma \ref{lem:sd.star.collapse}.

\begin{lem}
  \label{lem:cd.star.collapse}
  For all cubical sets $C$ and $v\in(\cd^2C)[0]$,
  \begin{equation*}
    (\gamma\bar\gamma)_{C}\Star_{\cd^2C}\langle v\rangle\subset\supp_{\cd^2}(\langle v\rangle,C).
  \end{equation*}
\end{lem}

We postpone a proof until \S\ref{subsubsec:proofs}. 

\subsection{Realizations}\label{subsec:cubical.realizations}
Geometric realization of cubical sets is the unique functor
\begin{equation*}
  |-|:\PRESHEAVES\BOX\ra\SPACES
\end{equation*}
sending $\BOX[0]$ to $\{0\}$, $\BOX[1]$ to the unit interval $\I$, each $\PRESHEAVES\BOX$-morphism $\BOX[\delta:[0]\ra[1]]:\BOX[0]\ra\BOX[1]$ to the map $\{0\}\ra\I$ defined by the $\BOX$-morphism $\delta:[0]\ra[1]$, finite tensor products to binary Cartesian products, and colimits to colimits.
We define open stars of geometric vertices as follows.  

\begin{defn}
  For each cubical set $C$ and subpresheaf $B\subset C$, we write
  $$\openstar_CB$$
  for the topological interior in $|C|$ of the subset $|\Star_CB|\subset|C|$ and call $\openstar_CB$ the \textit{open star of $B$} (\textit{in $C$}).  
\end{defn}

\begin{eg}
  As noted in \cite{fgr:ditop}, the set
  $$\{\openstar_B\langle v\rangle\}_{v\in B[0]}$$
  is an open cover of $|B|$ for each cubical set $B$.
\end{eg}

We abuse notation and write $\{0\}$ for the singleton space equipped with the unique possible circulation on it.  

\begin{defn}
  \label{defn:cubical.direalizations}
  We abuse notation and also write $\direalize{-}$ for the unique functor
  \begin{equation*}
    \direalize{\;-\;}\;:\PRESHEAVES\BOX\ra\STREAMS
  \end{equation*}
  sending $\BOX[0]$ to $\{0\}$, $\BOX[1]$ to $\vec{\I}$, each $\PRESHEAVES\BOX$-morphism $\BOX[\delta:[0]\ra[1]]:\BOX[0]\ra\BOX[1]$ to the stream map $\{0\}\ra\vec\I$ defined by the $\BOX$-morphism $\delta:[0]\ra[1]$, finite tensor products to binary Cartesian products, and colimits to colimits.
\end{defn}

We can henceforth identify the geometric realization $|B|$ of a cubical set $B$ with the underlying space of the stream $\direalize{B}$ because the forgetful functor $\STREAMS\ra\SPACES$ preserves colimits and Cartesian products [Proposition \ref{prop:topological}].  
In our proof of cubical approximation, we will need to say that a stream map $f:\direalize{B}\ra\direalize{D}$ whose underlying function corestricts to a subset of the form $|C|\subset|D|$ corestricts to a stream map $\direalize{B}\ra\direalize{C}$.  
In order to do so, we need the following observation.

\begin{thm}
  \label{thm:direalization.preserves.embeddings}
  The functor $\direalize{-}\;:\PRESHEAVES\BOX\ra\STREAMS$ sends monics to stream embeddings.
\end{thm}  

We give a proof at the end of \S\ref{sec:tri}.

\section{Triangulations}\label{sec:tri}
We would like to relate statements about simplicial sets to statements about cubical sets.  
In order to do so, we need to study properties of \textit{triangulation}, a functor $\tri:\PRESHEAVES\BOX\ra\PRESHEAVES\DEL$ decomposing each abstract $n$-cube into $n!$ simplices [Figure \ref{fig:sd.cd.tri}].  

\begin{defn}
  \label{defn:tri}
  We write $\tri$ for the unique cocontinuous functor 
  $$\tri:\PRESHEAVES\BOX\ra\PRESHEAVES\DEL$$
  naturally assigning to each cubical set $\BOX\boxobj{n}$ the simplicial set $\sn\boxobj{n}$.  
  We write $\qua$ for the right adjoint $\PRESHEAVES\DEL\ra\PRESHEAVES\BOX$ to $\tri$.  
\end{defn}

Triangulation $\tri$ restricts and corestricts to an isomorphism between full subcategories of cubical sets and simplicial sets having dimensions $0$ and $1$ because such cubical sets and simplicial sets are determined by their restrictions to $\OP\BOX_1$ and $\tri$ does not affect such restrictions.
The functor $\tri$ is cocontinuous by construction.  
Less straightforwardly, $\qua\circ\tri$ is also cocontinuous.  

\begin{lem}
  \label{lem:qt.cocontinuous}
  The composite $\qua\circ\tri:\hat\BOX\ra\hat\BOX$ is cocontinuous.
\end{lem}
\begin{proof}
  Let $\eta_B$ be the cubical function
  \begin{equation}
    \label{eqn:qt.cocontinuous}
    \eta_B:\int_{\BOX}^{\boxobj{n}}B(\boxobj{n})\cdot\qua\;(\tri\;\BOX\boxobj{n})\ra\qua\;(\tri B),
  \end{equation} 
  natural in cubical sets $B$, induced from all cubical functions of the form $\qua\;(\tri\;\theta_*):\qua\;(\tri\;\BOX\boxobj{n})\ra\qua\;(\tri B)$.
  Fix a cubical set $B$ and natural number $m$.  
  It suffices to show that $(\eta_B)_{\boxobj{m}}$, injective because $\qua\circ\tri$ preserves monics, is also surjective.  
  For then $\eta$ defines a natural isomorphism from a cocontinuous functor to $\qua\circ\tri:\PRESHEAVES\BOX\ra\PRESHEAVES\BOX$.  

  Consider a simplicial function $\sigma:\tri\BOX\boxobj{m}\ra\tri\;B$.  
  Consider a natural number $a$ and monotone function $\alpha:[a]\ra\boxobj{m}$ preserving extrema.
  The subpresheaf $\supp_{\tri}(\sigma(\sn\alpha)(\DEL[a]),B)$ of $B$ is atomic [Lemma \ref{lem:atomic.supports}] and hence there exist minimal natural number $n_\alpha$ and unique $\theta_\alpha\in B(\boxobj{n_\alpha})$ such that $\supp_{\tri}(\sigma(\sn\alpha)(\DEL[a]),B)=\langle\theta_\alpha\rangle$ [Lemma \ref{lem:atomic.supports}].  
  There exists a dotted monotone function $\lambda_\alpha:[a]\ra\boxobj{n_\alpha}$ such that the top trapezoid in the diagram
  \begin{equation*}
    \xymatrix{
      **[l]\DEL[a]
        \ar@{.>}[rrr]^-{\sn\lambda_{\alpha}}
        \ar[dd]|-{\sn\phi}
        \ar[dr]|-{\sn\alpha}
    &
    &
    & **[r]\tri\BOX\boxobj{n_{\alpha}}
        \ar[dl]|{\tri\;(\theta_{\alpha})_*}
      \ar@{.>}[dd]|{\tri\BOX[\delta_\phi]}
    \\
    & \tri\BOX\boxobj{m}
        \ar[r]|\sigma
    & \tri B 
    \\
      **[l]\DEL[b]
        \ar@{.>}[rrr]_{\sn\lambda_{\beta}}
        \ar[ur]|{\sn\beta}
   &
   &
   & **[r]\tri\BOX\boxobj{n_{\beta}}
       \ar[ul]|{\tri\;(\theta_{\beta})_*}
   }
\end{equation*}
  commutes by $\DEL[a]$ projective and $\sn$ full.  
  The isomorphism
  $$\BOX\boxobj{n_\alpha}/\partial\BOX\boxobj{n_\alpha}\cong \langle\theta_\alpha\rangle/\partial\langle\theta_\alpha\rangle$$
  induces an isomorphism
  $$\tri\BOX\boxobj{n_\alpha}/\tri\partial\BOX\boxobj{n_\alpha}\cong\tri\langle\theta_\alpha\rangle/\tri\partial\langle\theta_\alpha\rangle$$
  by $\tri$ cocontinuous and hence $(\tri\;(\theta_\alpha)_*)_{[a]}:(\tri\BOX\boxobj{n_\alpha})[a]\ra(\tri B)[a]$ is injective on $(\tri\BOX\boxobj{n_\alpha})[a]\setminus(\tri\partial\BOX\boxobj{n_\alpha})[a]$.  
  The function $\lambda_\alpha$ preserves extrema by minimality of $n_\alpha$ [Lemma \ref{lem:cubical.extrema}], hence $\lambda_\alpha\notin(\tri\partial\BOX\boxobj{n_\alpha})[a]$, hence the choice of $\lambda_\alpha$ is unique by the function $(\tri\;(\theta_\alpha)_*)_{[a]}$ injective on $(\tri\BOX\boxobj{n_\alpha})[a]\setminus(\tri\partial\BOX\boxobj{n_\alpha})[a]$.  
  
  We claim that our choices of $n_\alpha$ and $\theta_\alpha$ are independent of our choice of $a$ and $\alpha$.
  To check our claim, consider extrema-preserving monotone functions $\beta:[b]\ra\boxobj{n}$ and $\phi:[a]\ra[b]$ such that the left triangle commutes, with $n_\beta,\lambda_\beta,\theta_\beta$ defined analogously. 
  There exists a $\BOX$-morphism $\delta_\phi$ such that the right triangle commutes because $\BOX\boxobj{n_{\alpha}}$ is projective and the image of $(\theta_\alpha)_*$ lies in the image of $(\theta_\beta)_*$.
  The function $\delta_\phi$ preserves extrema because the outer square commutes and $\lambda_\alpha,\phi,\lambda_\beta$ preserve extrema.
  The function $\delta_\phi$ is injective by minimality of $n_\alpha$.
  Thus $\delta_\phi=\id_{\BOX\boxobj{n_\alpha}}$ [Lemma \ref{lem:cubical.extrema}].
  
  Let $\tau$ denote a monotone function from a non-empty ordinal to $\boxobj{m}$ preserving extrema.
  We have shown that all $n_\tau$'s coincide and all $\theta_{\tau}$'s coincide.
  Thus we can respectively define $N(\sigma)$ and $\Theta(\sigma)$ to be $n_\tau$ and $\theta_{\tau}$ for any and hence all choices of $\tau$.  
  The $\lambda_{\tau}$'s hence induce a simplicial function $\Sigma(\sigma):\tri\BOX\boxobj{m}\ra\tri\BOX\boxobj{N(\sigma)}$, well-defined by the uniquenesses of the $\lambda_{\tau}$'s, such that $(\tri\;\Theta(\sigma)_*)\circ\Sigma(\sigma)=\sigma$.  
  Thus the preimage of $\sigma$ under $(\eta_B)_{\boxobj{n}}$ is non-empty.  
\end{proof}

\begin{lem}
  \label{lem:qt.local.cn}
  There exists an isomorphism
  $$\qua\;(\tri B)\cong\int_{\BOX}^{\boxobj{n}}B(\boxobj{n})\cdot\cn\boxobj{n}$$
  natural in cubical sets $B$.  
\end{lem}
\begin{proof}
  There exist isomorphisms
  $$\qua\;(\tri\BOX\boxobj{n})=\qua\;(\sn\boxobj{n})=\cn\boxobj{n}$$
  natural in $\BOX$-objects $\boxobj{n}$.  
  The claim then follows by $\qua\circ\tri$ cocontinuous [Lemma \ref{lem:qt.cocontinuous}].  
\end{proof}

Triangulation relates our different subdivisions and hence justifies our abuse in notation for $\sd$, $\gamma$, and $\bar\gamma$.  
Figure \ref{fig:sd.cd.tri} illustrates a special case of the isomorphism claimed below.

\begin{prop}
  \label{prop:sd.cd.tri}
  There exists a dotted natural isomorphism making the diagram
  \begin{equation}
   \label{eqn:sd.cd.tri}
   \xymatrix@C=3pc{
     **[l]\tri
   & \tri\circ\cd
       \ar@{.>}[d]|\eta
       \ar[dr]^-{\tri\;\gamma}
       \ar[l]_-{\tri\;\bar\gamma}
   \\
   & \sd\circ\tri
       \ar[r]_-{\gamma\tri}
       \ar[ul]^-{\bar\gamma\tri}
   & **[r]\tri
   }
  \end{equation}
  commute.
\end{prop}
\begin{proof}
  The solid functions in the diagram
  \begin{equation}
   \label{eqn:sd.cd.tri.simple}
   \xymatrix@C=4pc{
     **[l]\BOX_1([m],[n])
   & \XBOX([m],\double\;[n])
     \ar@{.>}[d]|-{\alpha_{[m][n]}}
     \ar[dr]^-{\BOX(\id_{[m]},\gamma^{\BOX}_{[n]})}
     \ar[l]_-{\BOX(\id_{[m]},\bar\gamma^{\BOX}_{[n]})}
   \\
   & \DEL([m]^{\oplus 2},[n])
   \ar[r]_-{\DEL(\gamma^{\DEL}_{[m]},\id_{[n]})}
   \ar[ul]^-{\DEL(\bar\gamma^{\DEL}_{[m]},\id_{[n]})}
   & **[r]\BOX_1([m],[n]),
   }
  \end{equation}
  describe the $[m]$-components of the $\BOX[n]$-components of the solid natural transformations in (\ref{eqn:sd.cd.tri}) for the case $m,n\in\{0,1\}$.
  It suffices to construct a bijection $\alpha_{[m][n]}$, natural in $\BOX_1$-objects $[m]$ and $[n]$, making (\ref{eqn:sd.cd.tri.simple}) commute.
  For then the requisite cubical isomorphism $\eta_{B}$, natural in cubical sets $B$, in (\ref{eqn:sd.cd.tri}) would exist for the case $B=\BOX[n]$ for $\BOX_1$-objects $[n]$, hence for the case $B$ representable because all functors and natural transformations in $(\ref{eqn:sd.cd.tri})$ are monoidal (where we take $\PRESHEAVES\DEL$ to be Cartesian monoidal), and hence for the general case by naturality.  
  
  Let $\alpha_{[m][n]}$ and $\beta_{([m],[n])}$ be the functions
  $$\alpha_{[m][n]}:\XBOX([m],\double\;[n])\leftrightarrows\DEL([m]^{\oplus 2},[n]):\beta_{([m],[n])},$$
  natural in $\BOX_1$-objects $[m]$ and $[n]$, defined by
  \begin{align*}
    \alpha_{[m][n]}(\phi)(i)\;=\;&
    \begin{cases}
      (\gamma^{\BOX}_{[n]}\phi)(i), & i\in\{0,1,\ldots,m\}\\
      (\bar\gamma^{\BOX}_{[n]}\phi)(i-m-1), & i\in\{m+1,m+2,\ldots,2m+1\}
    \end{cases}\\
    \beta_{([m],[n])}(\phi)(j)\;=\;&\phi\gamma^{\DEL}_{[m]}(j)+\phi\bar\gamma^{\DEL}_{[m]}(j),\quad j=0,1,\ldots,m.
  \end{align*}
  An exhaustive check confirms that $\alpha_{[1][1]}(\phi):[3]\ra[1]$ is monotone for $\XBOX$-morphisms $\phi:[1]\ra[2]$. 
  An exhaustive check for all $m,n=0,1$ shows that $\alpha$ and $\beta$ are inverses to one another.
  Hence $\alpha$ defines a natural isomorphism in (\ref{eqn:sd.cd.tri.simple}).
  The right triangle in (\ref{eqn:sd.cd.tri.simple}) commutes because
  $$(\alpha_{[m][n]}(\phi))\gamma^{\DEL}_{[m]}=\alpha_{[m][n]}(\phi)_{\restriction[m]}=\gamma^{\BOX}_{[n]}\phi.$$
  Similarly, the left triangle in (\ref{eqn:sd.cd.tri.simple}) commutes.
\end{proof}

Triangulation relates our different stream realization functors.  

\begin{prop}
  \label{prop:tri.direalize}
  The following commutes up to natural isomorphism.
  \begin{equation*}
   \xymatrix@C=3pc{
     \hat\BOX
       \ar[r]^-{\direalize{\;-\;}}
    \ar[dr]_-{\tri}
   &
     \STREAMS
   \\
   & \hat\DEL
       \ar[u]_-{\direalize{\;-\;}}
   }
\end{equation*}
\end{prop}
\begin{proof}
  It suffices to show that there exists an isomorphism
  \begin{equation}
   \label{eqn:tri.direalize}
     \direalize{\DEL[n]}\;\cong\;\direalize{\BOX[n]}
  \end{equation}
  natural in $\BOX_1$-objects $[n]$ because $\direalize{-}:\hat\BOX\ra\STREAMS$ and $\tri$ send tensor products to binary Cartesian products, $\direalize{-}:\hat\DEL\ra\STREAMS$ preserves binary Cartesian products [Theorem \ref{thm:direalization.preserves.finite.products}], and colimits commute with tensor products in $\hat\BOX$ and $\STREAMS$ [Theorem \ref{thm:x.closed}].
  The linear homeomorphism $|\DEL[1]|\ra\I$ sending $|0|$ to $0$ and $|1|$ to $1$, an isomorphism of topological lattices and hence streams [Theorem \ref{thm:embed}], defines the $[1]$-component of our desired natural isomorphism (\ref{eqn:tri.direalize}) because
  $$\direalize{\BOX[[1]\ra[0]]}\;:\;\direalize{\BOX[1]}\ra\direalize{\BOX[0]},\quad\direalize{\DEL[[1]\ra[0]]}\;:\;\direalize{\DEL[1]}\ra\direalize{\DEL[0]}$$
  are both terminal maps and the functions $\direalize{\BOX[\delta:[0]\ra[1]]}\;:\;\direalize{\BOX[0]}\ra\direalize{\BOX[1]}$ and $\direalize{\DEL[\delta:[0]\ra[1]]}\;:\;\direalize{\DEL[0]}\ra\direalize{\DEL[1]}$ both send $0$ to minima or both send $0$ to maxima, for each function $\delta:[0]\ra[1]$. 
\end{proof}

\begin{defn}
  We write $\chomeo_B$ for isomorphism
  \begin{equation*}
    \chomeo_B\;:\;\direalize{\cd B}\;\cong\;\direalize{B},
  \end{equation*}
  natural in cubical sets $B$, induced from $\shomeo_{\tri B}:\;\direalize{\sd \tri B}\;\cong\;\direalize{\tri B}$ and the natural isomorphisms $\direalize{B}\cong\direalize{\tri B}$ and $\direalize{\tri\cd B}\cong\direalize{\sd\tri B}$ claimed in Propositions \ref{prop:sd.cd.tri} and \ref{prop:tri.direalize}.
\end{defn}

Context will make clear to which of the two natural isomorphisms
$$\shomeo:\direalize{\sd-}\cong\direalize{-}:\PRESHEAVES\DEL\ra\STREAMS,\quad\chomeo:\direalize{\cd-}\cong\direalize{-}:\PRESHEAVES\BOX\ra\STREAMS.$$
$\chomeo$ refers.

\subsubsection{Proofs of statements in \S\ref{sec:cubical}}\label{subsubsec:proofs}
The functor $\tri$ preserves and reflects monics and intersections of subobjects.
It follows that $\cd:\PRESHEAVES\BOX\ra\PRESHEAVES\BOX$, and hence also $\cd^2:\PRESHEAVES\BOX\ra\PRESHEAVES\BOX$, preserve monics and intersections of subobjects.
Moreover, $\sd:\PRESHEAVES\BOX\ra\PRESHEAVES\BOX$ preserves colimits by construction.  
Thus we can construct $\supp_{\cd^2}(B,C)\subset C$ for all cubical sets $C$ and subpresheaves $B\subset\cd^2C$.

\begin{proof}[Proof of Lemma \ref{lem:cd.fold}]
  For clarity, let $\psi_n$ denote the $\boxobj{n}$-component $\psi_{\boxobj{n}}$ of a cubical function $\psi$.  
  
  The last statement of the lemma would follow from the other statements because $A\cap\cd^2\partial B$ would be empty by minimality, the natural epi $\BOX\boxobj{\dim B}\ra B$ passes to an isomorphism $\BOX\boxobj{\dim B}/\partial\BOX\boxobj{\dim B}\cong B/\partial B$ and hence induces an isomorphism $\cd^2\BOX\boxobj{\dim B}/\cd^2\partial\BOX\boxobj{\dim B}\cong\cd^2B/\cd^2\partial B$ by $\cd$ cocontinuous, and all atomic subpresheaves of $\cd^2\BOX\boxobj{\dim B}$ are representable.  
  
  The case $C=\BOX[0]$ is immediate.
  
  The case $C=\BOX[1]$ follows from Lemma \ref{lem:sd.fold} because $\tri$ restricts and corestricts to an isomorphism between full subcategories of cubical sets and simplicial sets  having dimensions $0$ or $1$.
  
  The case $C$ representable then follows from an inductive argument.
  
  Consider the general case.  
  We can assume $C=\supp_{\cd^2}(A,C)$ without loss of generality and hence take $C$ to be atomic [Lemma \ref{lem:atomic.supports}].
  Let $\tilde{A}=\BOX\boxobj{\dim A}$ and $\tilde{C}=\BOX\boxobj{\dim C}$.  
  Let $\epsilon$ be the unique epi cubical function $\tilde{C}\ra C$. 
  We can identify $\tilde{A}$ with a subpresheaf of $\cd^2\tilde{C}$ and the unique epi $\tilde{A}\ra A$ with an appropriate restriction and corestriction of $\cd^2\epsilon:\cd^2\tilde{C}\ra\cd^2C$ by $\tilde{A}$ projective and $\dim A$ minimal.  
  There exist unique minimal subpresheaf $\tilde{B}\subset\tilde{C}$ such that $\tilde{A}\cap\cd^2\tilde{B}\neq\varnothing$ and unique retraction $\tilde\pi:\tilde{A}\ra\tilde{A}\cap\cd^2\tilde{B}$ by the previous case. 

  Let $B$ be the subpresheaf $\epsilon(\tilde{B})$ of $C$. 
  Consider a subpresheaf $B'\subset C$ such that $A\cap\cd^2B'\neq\varnothing$.  
  Pick an atomic subpresheaf $A'\subset A\cap\cd^2B'$.  
  Let $\tilde{A}'$ be an atomic subpresheaf of the preimage of $A'$ under the epi $\tilde{A}\ra A$.  
  Then $(\cd^2\epsilon)(\tilde{A}')\subset A'\subset\cd^2B'$, hence $\tilde{A}\cap(\cd^2\epsilon)^{-1}(B')=\tilde{A}\cap\cd^2(\epsilon^{-1}B')\neq\varnothing$, hence $\tilde{B}\subset\epsilon^{-1}B'$ by minimality of $\tilde{B}$, and hence $B\subset B'$.  
  
  Let $\gamma'=(\gamma\bar\gamma)_C$ and $\tilde\gamma'=(\gamma\bar\gamma)_{\tilde{C}}$.  
  
  It suffices to show that the cubical function $\pi:A\ra A\cap\cd^2B$ defined by
  \begin{equation}
    \label{eqn:cd.fold.retraction}
    \pi_n:\sigma\mapsto((\cd^2\epsilon)\tilde\pi)_n(\tilde\sigma),\quad n=0,1,\ldots,\;\sigma\in A(\boxobj{n}),\;\tilde\sigma\in(\cd^2\epsilon)_n^{-1}(\sigma)\cap\tilde{A}
  \end{equation}
  is well-defined. 
  For then, 
  \begin{align*}
    &\; \gamma'_n(\pi_n\sigma)\\
    =&\; \gamma'_n((\cd^2\epsilon)_n(\tilde\pi_n\tilde\sigma)),\;\tilde\sigma\in(\cd^2\epsilon)^{-1}_n(\sigma)&& \text{definition of $\pi$}\\
    =&\; (\cd^2\epsilon)_n(\tilde\gamma'_n(\tilde\pi_n\tilde\sigma)),\;\tilde\sigma\in(\cd^2\epsilon)^{-1}_n(\sigma) && \text{naturality of $\gamma\bar\gamma$}\\
    =&\; (\cd^2\epsilon)_n(\tilde\gamma'_n\tilde\sigma),\;\tilde\sigma\in(\cd^2\epsilon)^{-1}_n(\sigma) && \text{previous case}\\
    =&\; \gamma'_n(\sigma) && \text{naturality of $\gamma\bar\gamma$}
  \end{align*}
  
  We show $\pi$ is well-defined by induction on $\dim C$.  
  
  In the base case $C=\BOX[0]$, $B=\BOX[0]$ and hence $\pi$ is the well-defined terminal cubical function.  

  Consider a natural number $d$, inductively assume $\pi$ is well-defined for the case $\dim C<d$, and consider the case $\dim C=d$.  
  Consider natural number $n$, $\sigma\in A(\boxobj{n})$, and $\tilde\sigma\in(\cd^2\epsilon)_n^{-1}\sigma\cap\tilde{A}$.  

  In the case $\sigma\notin\cd^2\partial C$, $(\cd^2\epsilon)_n^{-1}\sigma=\{\tilde\sigma\}$ because $\cd^2\epsilon:\cd^2\tilde{C}\ra\cd^2C$ passes to an isomorphism $\cd^2\tilde{C}/\cd^2\partial\tilde{C}\ra\cd^2C/\cd^2\partial C$ by $\cd^2$ cocontinuous.  
  Hence $\pi_n(\tilde\sigma)$ is well-defined.

  Consider the case $\sigma\in\cd^2\partial C$ and hence $\tilde\sigma\in\cd^2\partial\tilde C$.  
  Then $\tilde{B}$ is the minimal subpresheaf of $\tilde{C}$ such that $\langle\tilde\sigma\rangle\cap\cd^2\tilde{B}\neq\varnothing$ by $\tilde{B}$ minimal and hence the unique retraction $\tilde\pi_{\sigma}:\langle\tilde\sigma\rangle\ra\langle\tilde\sigma\rangle\cap\cd^2\tilde{B}$ is a restriction of $\tilde\pi$.    
  The retraction $\tilde\pi_{\sigma}$ passes to a well-defined retraction $\pi_{\sigma}:\langle\sigma\rangle\ra\langle\sigma\rangle\cap\cd^2 B$ by the inductive hypothesis because $\dim\supp_{\cd^2}(\langle\sigma\rangle,C)\leq\dim\supp_{\cd^2}(\cd^2\partial C,C)=\dim\partial C=d-1$.  
  Thus $((\cd^2\epsilon)\tilde\pi)_n(\tilde\sigma)=((\cd^2\epsilon)\tilde\pi_{\sigma})_n(\tilde\sigma)=(\pi_{\sigma})_n(\sigma)$ is a function of $\sigma$ and hence $\pi_n(\sigma)$ is well-defined.  

\end{proof}

\begin{proof}[Proof of Lemma \ref{lem:cd.star.collapse}]
  Consider an atomic subpresheaf $A\subset\cd^2 C$ such that $v\in A[0]$.  
  There exists a minimal subpresheaf $B\subset C$ such that $A\cap\cd^2B\neq\varnothing$ [Lemma \ref{lem:cd.fold}].  
  Hence $B\subset\langle\supp_{\cd^2}(\langle v\rangle,C)\rangle$ by minimality.  
  Hence
  $$(\gamma\bar\gamma)_C(A)\subset(\gamma\bar\gamma)_C(A\cap\cd^2B)\subset(\gamma\bar\gamma)_C(\cd^2B)\subset B\subset\supp_{\cd^2}(\langle v\rangle,C)$$
  by Lemma \ref{lem:cd.fold}.  
\end{proof}

We introduce the following lemma as an intermediate step in proving Theorem \ref{thm:direalization.preserves.embeddings}.
The lemma asserts that we can approximate the global preorder $\leqslant_{\direalize{B}}$ on a stream $\direalize{\;B\;}$ by the value of the circulation of $\direalize{C}$ on smaller and smaller neighborhoods of $|B|$ in $|C|$.
Such neighborhoods take the form of open stars of $B$ in $C$ after taking iterated subdivisions of $C$.

\begin{lem}
  \label{lem:direalization.preserves.preordered.embeddings}
  For each pair of cubical sets $B\subset C$,
  \begin{equation}
    \label{eqn:embeddings}
    \graph{\leqslant^{\direalize{\;B\;}}_{|B|}}=|B|^2\cap\bigcap_{n=1}^\infty\;\graph{\leqslant^{\direalize{\;C\;}}_{\chomeo^n_C\openstar_{\cd^nC}\cd^nB}}
  \end{equation}
  where $\leqslant^X$ denotes the circulation of a stream $X$.
\end{lem}
\begin{proof}
  Let $n$ denote a natural number.
  For each $n$,
  \begin{equation*}
    \label{eqn:LHS.embeddings}
    \graph{\leqslant^{\direalize{\;\cd^nB\;}}_{|\cd^nB|}}\subset\graph{\leqslant^{\direalize{\;\Star_{\cd^nC}\cd^nB\;}}_{\openstar_{\cd^nC}\cd^nB}}\subset\graph{\leqslant^{\direalize{\;\cd^nC\;}}_{\openstar_{\cd^nC}\cd^nB}}
  \end{equation*}
  because inclusions define stream maps of the form
  $$\direalize{\cd^nB\ira\Star_{\cd^nC}\cd^nB},\quad\direalize{\Star_{\cd^nC}\cd^nB\ira\cd^nC}$$
  In (\ref{eqn:embeddings}), the inclusion of the left side into the right side follows because the $\ordinomorphism^n$'s define natural isomorphisms $\direalize{\cd^nB}\cong\direalize{B}$ and $\direalize{\cd^nC}\cong\direalize{C}$.

  Consider $(x,y)\in|B|^2$ not in the left side of (\ref{eqn:embeddings}).
  We must show that $(x,y)$ is not in the right side of (\ref{eqn:embeddings}).  
  
  Let $\mathcal{B}$ be the collection of finite subpresheaves of $B$.  
  The underlying preordered set of $\direalize{B}$ is the filtered colimit of underlying preordered sets of streams $\direalize{A}$ for $A\in\mathcal{B}$ by the forgetful functor $\STREAMS\ra\PREORDEREDSETS$ cocontinuous [Proposition \ref{prop:almost.topological}].  
  Hence
  $$\graph{\leqslant_{\direalize{\;B\;}}}=\bigcup_{A\in\mathcal{B}}\graph{\leqslant_{\direalize{\;A\;}}}.$$ 
  It therefore suffices to take the case $B$ finite.
  In particular, we can take $|B|$ to be metrizable.  

  There exists a neighborhood $U\times V$ in $|B|^2$ of $(x,y)$ such that
  \begin{equation}
    \label{eqn:closed.preorder.neighborhood}
  (U\times V)\cap\graph{\leqslant^{\direalize{\;B\;}}_{|B|}}=\varnothing
  \end{equation}
  by $\graph{\leqslant^{\direalize{\;B\;}}_{|B|}}$ closed in $|B|^2$ [Lemma \ref{lem:simplicial.preorders} and Proposition \ref{prop:tri.direalize}].
  Let
  $$A_n(z)=\supp_{|-|}(\ordinomorphism^{-n}_B(z),\cd^nB),\quad z\in|B|,\;n=0,1,\ldots$$
  The $A_n(z)$'s are atomic [Lemma \ref{lem:atomic.supports}] and hence the diameters of the $\ordinomorphism^n_B|A_n(z)|$'s, for each $z$, approach $0$ as $n\ra\infty$ under a suitable metric on $|B|$.  
  Fix $n\gg 0$.  
  Then $\ordinomorphism^{n-2}_B|A_{n-2}(x)|\subset U$ and $\ordinomorphism^{n-2}_B|A_{n-2}(y)|\subset V$.
  The cubical function $(\gamma\bar\gamma)_{\cd^{n-2}C}:\cd^nC\ra\cd^{n-2}C$ restricts and corestricts to a cubical function 
  $$\gamma'_{(n)}:\Star_{\cd^nC}\cd^nB\subset\cd^{n-2}B$$
  such that $\gamma'_{(n)}\Star_{\cd^nC}A_n(z)\subset A_{n-2}(z)$ for all $z\in|B|$ [Lemma \ref{lem:cd.star.collapse}].
  Then
  $$\ordinomorphism^{n-2}_B\direalize{\gamma'_n}(x')\nleqslant^{\direalize{\;B\;}}_{|B|}\ordinomorphism^{n-2}_B\direalize{\gamma'_n}(y'),\quad x'\in\openstar_{\cd^nC}A_n(x),\;y'\in\openstar_{\cd^nC}A_n(y)$$
  by (\ref{eqn:closed.preorder.neighborhood}). 
  Hence by $\ordinomorphism^{n-2}_B\direalize{\gamma'_{(n)}}\;:\;\direalize{\Star_{\cd^nC}\cd^nB}\;\ra\;\direalize{B}$ a stream map, 
  \begin{equation}
    x'\nleqslant^{\direalize{\;C'\;}}_{\openstar_{\cd^nC}\cd^nB}y',\quad x'\in\openstar_{\cd^nC}A_n(x),\;y'\in\openstar_{\cd^nC}A_n(y)
    \label{eqn:almost.RHS.embeddings}
  \end{equation}
  for $C'=\Star_{\cd^nC}\cd^nB$ and hence $C'=\cd^nC$ because $\leqslant^{\direalize{\;B''\;}}_U=\leqslant^{\direalize{\;C''\;}}_U$ for all open subsets $U$ of $|C''|$ which happen to lie in $|B''|$ for all cubical sets $B''\subset C''$ by an application of Proposition \ref{prop:almost.topological}.
  
  We conclude the desired inequality
  $$x\nleqslant^{\direalize{\;C\;}}_{\ordinomorphism^n_C\openstar_{\cd^nC}\cd^nB}y$$
  by setting $(x',y')=(\ordinomorphism^{-n}_Cx,\ordinomorphism^{-n}_Cy)$ and applying $\ordinomorphism^{n}_C:\direalize{\cd^{n}C}\cong\direalize{C}$ to (\ref{eqn:almost.RHS.embeddings}).
\end{proof}

We can now prove that stream realizations of cubical sets preserve embeddings.  

\begin{proof}[Proof of Theorem \ref{thm:direalization.preserves.embeddings}]
  Consider an object-wise inclusion
  $$\iota:B\ira C.$$
  
  Take the case $B$ finite.
  Consider $x\in|B|$.  
  Let
  $$B(x,m)=\Star_{\cd^mB}(\supp_{|-|}(\ordinomorphism^{-m}_B(x),\cd^mB)),\quad m=0,1,\ldots.$$
  The $\ordinomorphism^m_BB(x,m)$'s form a neighborhood basis of $x$ in $|B|$ by $B$ finite.  
  And
  \begin{equation*}
    \graph{\leqslant^{\direalize{\;B(x,m)\;}}_{|B(x,m)|}}=|B(x,m)|^2\cap\bigcap_{n=1}^\infty\;\graph{\leqslant^{\direalize{\;\sd^mC\;}}_{\ordinomorphism^{n}_{\cd^mC}|\openstar_{\cd^{m+n}C}\sd^nB(x,m)|}}
  \end{equation*}
  [Lemma \ref{lem:direalization.preserves.preordered.embeddings}].  
  An application of Lemma \ref{lem:restrictions.define.subcirculations} allows us to conclude that $\direalize{\iota}\;:\;\direalize{B}\ra\direalize{C}$ is a stream embedding.

  Take the general case.  
  Consider a stream embedding $k:K\ra\direalize{B}$ from a compact Hausdorff stream $K$.  
  Let $B'=\direalize{\supp_{|-|}(k(K),B)}$.  
  Inclusion defines a stream embedding $\direalize{B'\ira B}\;:\;\direalize{B'}\ra\direalize{B}$ by the previous case and hence the stream map $k$ corestricts to a stream embedding $k':K\ra\direalize{B'}$ by universal properties of stream embeddings.
  The composite $\direalize{B'\ira B\ira C}\;:\;\direalize{B'}\ra\direalize{C}$ is a stream embedding by the previous case.  
  Thus the composite of stream embedding $k':K\ra\direalize{B'}$ and stream embedding $\direalize{B'\ra C}\;:\;\direalize{B'}\ra\direalize{C}$, is a stream embedding $K\ra\direalize{C}$.
  Thus $\direalize{\iota}\;:\;\direalize{B}\ra\direalize{C}$ is a stream embedding [Lemma \ref{lem:k-embeddings}].  
\end{proof}

\section{Homotopy}\label{sec:equivalence}
Our goal is to prove an equivalence between combinatorial and topological homotopy theories of directed spaces, based directed spaces, pairs of directed spaces, and more general diagrams of directed spaces.  
We can uniformly treat all such variants of directed spaces as functors to categories of directed spaces.
Let $\GENERIC$ denote a category and $\DIAGRAM$ denote a small category throughout \S\ref{sec:equivalence}.  

\begin{defn}
  Fix $\GENERIC$.  
  A \textit{$\GENERIC$-stream} is a functor
  $$\GENERIC\ra\STREAMS$$ 
  and a \textit{$\GENERIC$-stream map} is a natural transformation between $\GENERIC$-streams.  
  We similarly define \textit{$\GENERIC$-simplicial sets}, \textit{$\GENERIC$-simplicial functions}, \textit{$\GENERIC$-cubical sets}, and \textit{$\GENERIC$-cubical functions}.
\end{defn}

We describe $\DIAGRAM$-objects in terms of their coproducts as follows.

\begin{defn}
  Fix $\DIAGRAM$.
  A $\DIAGRAM$-stream $X$ is \textit{compact} if the $\STREAMS$-coproduct
  $$\coprod X$$
  is a compact stream.
  We similarly define \textit{finite} $\DIAGRAM$-simplicial sets, \textit{finite} $\DIAGRAM$-cubical sets, (\textit{open}) \textit{$\DIAGRAM$-substreams} of $\DIAGRAM$-streams, \textit{open covers} of $\DIAGRAM$-streams by $\DIAGRAM$-substreams, and \textit{$\DIAGRAM$-subpresheaves} of $\DIAGRAM$-cubical sets and $\DIAGRAM$-simplicial sets.
\end{defn}

Constructions on directed spaces naturally generalize.  
For each $\GENERIC$ and $\GENERIC$-simplicial set $B$, we write $\direalize{B}$ for the $\GENERIC$-stream $\direalize{-}\circ\;B$.  
We make similar such abuses of notation and terminology throughout \S\ref{sec:equivalence}.
For the remainder of \S\ref{sec:equivalence}, let $\DIAGRAM$ be a fixed small category and $g$ denote a $\DIAGRAM$-object.  

\subsection{Streams}\label{subsec:stream.homotopy}
We adapt a homotopy theory of directed spaces \cite{grandis:d} for streams. 

\begin{defn}
  Fix $\GENERIC$. 
  Consider a $\GENERIC$-stream $X$.
  We write $i_0,i_1$ for the $\GENERIC$-stream maps $X\ra X\times_\STREAMS\vec\I$ defined by $(i_0)_c(x)=(x,0)$, $(i_1)_c(x)=(x,1)$ for each $\GENERIC$-object $c$, when $X$ is understood.  
  Consider $\GENERIC$-stream maps $f,g:X\ra Y$.
  A \textit{directed homotopy from $f$ to $g$} is a dotted $\GENERIC$-stream map making
  \begin{equation*}
    \xymatrix@C=4pc{
    **[l]X\coprod_{\STREAMS}X\ar[r]^{f\coprod_{\STREAMS}g}\ar[d]_{i_0\coprod_{\STREAMS}i_1}
  & **[r]Y\\
  **[l]X\times_{\STREAMS}\vec\I\ar@{.>}[ur]
    }
  \end{equation*}
  commute.  
  We write $f\futurehomotopic g$ if there exists a directed homotopy from $f$ to $g$.  
  We write $\homotopic$ for the equivalence relation on $\GENERIC$-stream maps generated by $\futurehomotopic$.
\end{defn}

We give a criterion for stream maps to be $\homotopic$-equivalent.  
We review definitions of lattices and lattice-ordered topological vector spaces in Appendix \S\ref{sec:preorders}.  
A function $f:X\ra Y$ between convex subspaces of real topological vector spaces is \textit{linear} if $tf(x')+(1-t)f(x'')=f(tx'+(1-t)x'')$ for all $x',x''\in X$ and $t\in\I$.    

\begin{lem}
  \label{lem:primordial.homotopy}
  Fix $\GENERIC$.  
  Consider a pair of $\GENERIC$-stream maps
  $$f',f'':X\ra Y.$$
  If $Y(g)$ is a compact Hausdorff connected, convex subspace and sublattice of a lattice-ordered topological vector space for each $\DIAGRAM$-object $g$ and $Y(\gamma)$ is a linear lattice homomorphism $Y(g')\ra Y(g'')$ for all $\DIAGRAM$-morphisms $\gamma:g'\ra g''$, then $f'\homotopic f''$.  
\end{lem}
\begin{proof}
  Let $f'\join f''$ be the $\GENERIC$-stream map $X\ra Y$ defined by
  $$(f'\join f'')_g(x)=f'_g(x)\join_{Y(g)}f''_g(x).$$
  Linear interpolation defines directed homotopies $f'\futurehomotopic f'\join f''$ and $f''\futurehomotopic f'\join f''$.  
\end{proof}

\begin{defn}
  We write $\HOMOTOPY\EQUISTREAMS{\DIAGRAM}$ for the quotient category
  $$\HOMOTOPY\STREAMS^{\DIAGRAM}=\STREAMS^{\DIAGRAM}/\homotopic.$$
\end{defn}

\subsection{Cubical sets}\label{subsec:cubical.homotopy}
We define an analogous homotopy theory for cubical sets.  

\begin{defn}
  Consider a pair $\alpha,\beta:B\ra C$ of $\DIAGRAM$-cubical functions.  
  A \textit{directed homotopy from $\alpha$ to $\beta$} is a dotted $\DIAGRAM$-cubical function making
  \begin{equation*}
    \xymatrix@C=4pc{
    **[l]B\coprod_{\STREAMS}B\ar[r]^{\alpha\coprod_{\PRESHEAVES\BOX}\beta}\ar[d]_{-\otimes\BOX[\delta_-]\coprod_{\PRESHEAVES\BOX}-\otimes\BOX[\delta_+]}
  & **[r]C\\
    **[l]B\otimes\BOX[1]\ar@{.>}[ur]
    }
  \end{equation*}
  commute.  
  We write $\alpha\futurehomotopic\beta$ if there exists a directed homotopy from $\alpha$ to $\beta$.  
  We write $\homotopic$ for the equivalence relation on $\DIAGRAM$-cubical functions generated by $\futurehomotopic$.  
\end{defn}

Cubical nerves of semilattices are naturally homotopically trivial in the following sense.  
Recall that tensor products $B\otimes C$ of cubical sets naturally reside as subpresheaves of binary Cartesian products $B\times C$ and hence admit projections $B\otimes C\ra B$ and $B\otimes C\ra C$.  
Recall definitions of \textit{semilattices} and \textit{semilattice homomorphisms} in Appendix \S\ref{sec:preorders}.  

\begin{lem}
  \label{lem:convex.nerves}
  The two $\SEMILATTICES$-cubical functions
  \begin{equation}
    \label{eqn:convex.nerves}
    (\cn_{\restriction\SEMILATTICES})^{\otimes 2}\ra\cn_{\restriction\SEMILATTICES}:\SEMILATTICES\ra\PRESHEAVES\BOX
  \end{equation}
  defined by projection onto first and second factors are $\homotopic$-equivalent, where $\SEMILATTICES$ is the category of semilattices and semilattice homomorphisms.  
\end{lem}
\begin{proof}
  Let $S$ be the $\SEMILATTICES$-cubical set $\cn_{\restriction\SEMILATTICES}$.
  Let $\pi_1,\pi_2$ be the $\SEMILATTICES$-cubical functions $S^{2}\ra S$ defined by projection onto first and second factors.
  It suffices to show $\pi_1\homotopic\pi_2$ because both projections of the form $S^{\otimes 2}\ra S$ are restrictions of $\pi_1,\pi_2:S^2\ra S$.

  The function $\eta_{X}:X^2\times[1]\ra X$, natural in $\SEMILATTICES$-objects $X$ and defined by
  $$\eta_{X}(\epsilon_1,\epsilon_2,\varepsilon)=\begin{cases}\epsilon_1\meet_{X}\epsilon_2,&\varepsilon=0\\ \epsilon_1,&\varepsilon=1\end{cases},$$
  is monotone. 
  Hence we can construct a directed homotopy
  $$(\cn\eta)_{\restriction S^2\otimes\BOX[1]}:S^{2}\otimes\BOX[1]\ra S.$$
  from $\cn\eta(-_1,-_2,0)$ to $\pi_1$.
  Similarly $\cn\eta(-_1,-_2,0)\futurehomotopic\pi_2$.
\end{proof}

\begin{defn}
  We write $\HOMOTOPY\PRESHEAVES\BOX^{\DIAGRAM}$ for the quotient category
  $$\HOMOTOPY\PRESHEAVES\BOX^{\DIAGRAM}=\PRESHEAVES\BOX^{\DIAGRAM}/\homotopic.$$
\end{defn}

\subsection{Main results}\label{subsec:main}
We prove that the directed homotopy theories of streams and cubical sets, defined previously, are equivalent.
Thus we argue that directed homotopy types describe tractable bits of information about the dynamics of processes from their streams of states.  
We fix a small category $\DIAGRAM$.  

\begin{lem}
  \label{lem:tri.equivalence}
  The $\PRESHEAVES\BOX$-stream maps
  $$\direalize{\epsilon_{\tri}}\;:\;\direalize{\tri\circ\qua\circ\tri -}\;\leftrightarrows\;\direalize{\tri-}\;:\;\direalize{\tri\;\eta}$$
  are inverses up to $\homotopic$, where $\eta$ and $\epsilon$ are the respective unit and counit of the adjunction $\tri\vdash\qua$.  
\end{lem}
\begin{proof}
  The left hand side is a retraction of the right hand side by a zig-zag identity.
  It therefore suffices to show that the right side is a left inverse to the left side up to $\homotopic$.  

  Both projections of the form $(\cn_{\restriction\BOX})^{\otimes 2}\ra\cn_{\restriction\BOX}$ are $\homotopic$-equivalent as $\BOX$-cubical functions [Lemma \ref{lem:convex.nerves}], hence both projections of the form $\qua\circ\tri\circ\BOX[-]^{\otimes 2}\ra\qua\circ\tri\circ\BOX[-]$ are $\homotopic$-equivalent by $\qua\circ\tri\circ\BOX[-]\cong\cn_{\restriction\BOX}$, hence both projections of $\BOX$-streams of the form
  $$\direalize{\tri\circ\qua\circ\tri\circ\BOX[-]}^2\ra\direalize{\tri\circ\qua\circ\tri\circ\BOX[-]}$$
  are $\homotopic$-equivalent because $\direalize{-}:\PRESHEAVES\BOX\ra\STREAMS$ sends tensor products to binary Cartesian products, hence all $\BOX$-stream maps from the same domain to $\direalize{\tri\circ\qua\circ\tri\circ\BOX[-]}$ are $\homotopic$-equivalent, hence $\direalize{\tri\;\eta_{\BOX[-]}}$ is a left inverse to $\direalize{\epsilon_{\tri\BOX[-]}}$ up to $\homotopic$, and hence $\direalize{\tri\eta}$ is a left inverse to $\direalize{\epsilon_{\tri}}$ up to $\homotopic$ [Lemma \ref{lem:qt.local.cn}]. 
\end{proof}

The second lemma allows us to naturally lift certain stream maps, up to an approximation, to hypercubes.

\begin{lem}
  \label{lem:bamfl}
  Consider $\DIAGRAM$-cubical set $C$ and $\DIAGRAM$-stream map
  $$f:X\ra\direalize{\cd^4C}$$
  such that $X(g)\neq\varnothing$ and $f_g(X(g))$ lies in an open star of a vertex of $\cd^4C(g)$ for each $g$.  
  There exist 
  \begin{enumerate}
    \item functor $R:\DIAGRAM\ra\BOX$
    \item object-wise monic $\DIAGRAM$-cubical function $\iota:\BOX[R-]\ra\cd^2C$
    \item $\DIAGRAM$-stream map $f':X\ra\direalize{\BOX[R-]}$
  \end{enumerate}
  such that $\direalize{(\bar\gamma\gamma)_C\iota}f'=\direalize{(\gamma\bar\gamma)^2_C}f$.
\end{lem}
\begin{proof}
  Let $S$ be the $\DIAGRAM$-subpresheaf of $\cd^2C$ such that for each $g$,
  $$S(g)=\supp_{|-|}(|(\gamma\bar\gamma)_{\cd^2C(g)}|f(X(g)),\cd^2C(g)).$$  
  
  Fix $g$.
  There exists a minimal atomic subpresheaf $A(g)$ of $\cd^2C(g)$ containing $S(g)$ [Lemma \ref{lem:cd.star.collapse}], and hence there exist minimal subpresheaf $B(g)$ of $C(g)$ such that $A(g)\cap\cd^2B(g)\neq\varnothing$ and unique retraction $A(g)\ra A(g)\cap\cd^2B(g)$ [Lemma \ref{lem:cd.fold}].  
  We claim that $A(g)\cap\cd^2B(g)$ is independent of our choice of $A(g)$. 
  To see the claim, it suffices to consider the case that there exist distinct possible choices $A'(g)$ and $A''(g)$ of $A(g)$; then $A'(g)\cap A''(g)=A'(g)\cap\cd^2\supp_{\cd^2}(A''(g),C(g))=A''(g)\cap\cd^2\supp_{\cd^2}(A'(g),C(g))$, and hence $B(g)$ is independent of the choices $A(g)=A'(g),A''(g)$ and $A'(g)\cap\cd^2B(g)=A''(g)\cap\cd^2B(g)=A'(g)\cap A''(g)\cap\cd^2B(g)$ by $B(g)$ minimal.

  The assignment $g\mapsto A(g)\cap\cd^2B(g)$ extends to a functor $A\cap\cd^2B:\DIAGRAM\ra\PRESHEAVES\BOX$ such that suitable restrictions of retractions $A(g)\ra A(g)\cap\cd^2B(g)$ define a $\DIAGRAM$-cubical function $\pi:S\ra A\cap\cd^2B$ by an application of Lemma \ref{lem:natural.cd.folding}.
  The $\DIAGRAM$-cubical set $A\cap\cd^2B$ is of the form $\BOX[R-]$ up to natural isomorphism [Lemma \ref{lem:cd.fold}].  
  The $\DIAGRAM$-function $|(\gamma\bar\gamma)_{\cd^2C}|f:X\ra|\cd^2C|$ corestricts to a function of the form $X\ra|S|$ [Lemma \ref{lem:cd.fold}] and hence the $\DIAGRAM$-stream map $\direalize{(\gamma\bar\gamma)_{\cd^2C}}f:X\ra\direalize{\cd^2C}$ corestricts to a $\DIAGRAM$-stream map $f'':X\ra\direalize{S}$ [Theorem \ref{thm:direalization.preserves.embeddings}].  
  The object-wise inclusion $A\cap\cd^2B\ira\cd^2C$ and $\DIAGRAM$-stream map $\direalize{\pi}f''$ are isomorphic to our desired $\iota$ and $f'$.  
\end{proof}

\begin{prop}
  \label{prop:sd.natural.approx}
  The following $\PRESHEAVES\DEL$-stream maps are $\homotopic$-equivalent.
  \begin{equation*}
    \direalize{\gamma}\;\homotopic\;\direalize{\bar\gamma}\;\homotopic\shomeo:\;\direalize{\sd-}\ra\direalize{-}\;:\PRESHEAVES\DEL\ra\STREAMS.
  \end{equation*}
\end{prop}
\begin{proof}
  The $\DEL$-stream maps
  $$\direalize{\gamma_{\DEL[-]}}\homotopic\direalize{\bar\gamma_{\DEL[-]}}\homotopic\shomeo_{\DEL[-]}:\DEL\ra\STREAMS$$
  by Lemma \ref{lem:primordial.homotopy}.  
  The general claim follows from naturality.
\end{proof}

A cubical analogue follows [Propositions \ref{prop:sd.cd.tri} and \ref{prop:sd.natural.approx}].

\begin{cor}
  \label{cor:cd.natural.approx}
  The following $\PRESHEAVES\BOX$-stream maps are $\homotopic$-equivalent.
  \begin{equation*}
    \direalize{\gamma}\;\homotopic\;\direalize{\bar\gamma}\;\homotopic\chomeo:\;\direalize{\cd-}\ra\direalize{-}\;:\PRESHEAVES\BOX\ra\STREAMS.
  \end{equation*}
\end{cor}

Other directed homotopy theories in the literature \cite{fgr:ditop} use cylinder objects defined in terms of the unit interval equipped with the trivial circulation instead of $\direalize{\BOX[1]}$.
While the latter homotopy relation is generally weaker than the former homotopy relation $\homotopic$ that we adopt, we identify criteria for both relations to coincide.
Figure \ref{fig:homotopy} illustrates the difference between the two homotopy relations.

\begin{thm}
  \label{thm:d}
  The following are equivalent for a pair of $\DIAGRAM$-stream maps
  $$f',f'':X\ra Y$$
  from a compact $\DIAGRAM$-stream $X$ to a quadrangulable $\DIAGRAM$-stream $Y$.
  \begin{enumerate}
    \item\label{item:d-homotopic} $f'\homotopic f''$.   
    \item\label{item:dihomotopic} There exists $\DIAGRAM$-stream map $H:X\times_{\STREAMS}\I\ra Y$ such that $H(-,0)=f'$ and $H(-,1)=f''$, where we regard the unit interval $\I$ as equipped with the circulation trivially preordering each open neighborhood.
  \end{enumerate}
\end{thm}
\begin{proof}
  The implication (\ref{item:d-homotopic})$\Rightarrow$(\ref{item:dihomotopic}) follows because directed homotopies of $\DIAGRAM$-stream maps $X\ra Y$ define $\DIAGRAM$-stream maps of the form $X\times\I\ra Y$.  
  
  Suppose (\ref{item:dihomotopic}).  
  We take, without loss of generality, $Y$ to be of the form $\direalize{\cd^4C}$ for a $\DIAGRAM$-cubical set $C$.
  There exists a finite collection $\mathscr{O}$ of open $\DIAGRAM$-substreams $U$ of $X$, natural number $k$, and finite sequence $0=t_0<t_1<t_2<\cdots<t_{k-1}<t_k=1$ of real numbers such that $H(U(g)\times[t_i,t_{i+1}])$ lies inside an open star of a vertex of $\cd^4C(g)$ for each $\DIAGRAM$-object $g$ and each $U\in\mathscr{O}$ by $X$ compact.
 
  It suffices to consider the case $k=1$; the general case would then follow from induction.  
  We take $\mathscr{O}$ to be closed under non-empty intersection without loss of generality and regard $\mathscr{O}$ as a poset ordered by inclusion.  
  We take $(\mathscr{O}\times\DIAGRAM)'$ to be the full subcategory of $\mathscr{O}\times\DIAGRAM$ consisting of pairs $(U,g)$ with $U(g)\neq\varnothing$.
  Let $X'$ be the $(\mathscr{O}\times\DIAGRAM)'$-stream naturally sending each pair $(U,g)$ to $U(g)$.   
  Let $C'$ be the $(\mathscr{O}\times\DIAGRAM)'$-cubical set naturally sending each pair $(U,g)$ to $C(g)$.   
  Let $H'$ be the $(\mathscr{O}\times\DIAGRAM)'$-stream map $X'\times_{\STREAMS}\I\ra\direalize{\cd^4C'}$ defined by suitable restrictions of $H_g$'s.  
  There exist functor $R:(\mathscr{O}\times\DIAGRAM)'\ra\BOX$, object-wise monic $(\mathscr{O}\times\DIAGRAM)'$-cubical function $\iota:\BOX[R-]\ra\cd^2C'$, and $(\mathscr{O}\times\DIAGRAM)'$-stream map $H'':X'\times_{\STREAMS}\I\ra\direalize{\BOX[R-]}$ such that $\direalize{(\bar\gamma\gamma)_{C'}\iota}H''=\direalize{(\gamma\bar\gamma)^2_{C'}}H'$ [Lemma \ref{lem:bamfl}].  
  Then $H''(-,0)\homotopic H''(-,1)$ [Lemma \ref{lem:primordial.homotopy}], hence $\direalize{(\bar\gamma\gamma)^2_C}H'(-,0)\homotopic\direalize{(\bar\gamma\gamma)^2_C}H'(-,1)$, hence $\direalize{(\bar\gamma\gamma)^2_C}f'\homotopic \direalize{(\bar\gamma\gamma)^2_C}f''$ by taking colimits, and hence $f'\homotopic f''$ [Corollary \ref{cor:cd.natural.approx}].  
\end{proof}

We now prove our main simplicial and cubical approximation theorems.  

\begin{thm}
\label{thm:sd.approx}
  Consider a commutative diagram on the left side of
\begin{equation}
   \label{eqn:approx}
   \xymatrix@C=3pc{
     **[l]\direalize{B}
       \ar[d]_-{\direalize{\;\beta\;}}
       \ar[r]^-{\direalize{\;\alpha\;}}
   & **[r]\direalize{\tri\;D}
   & **[l]\sd^kB
       \ar[r]^-{(\gamma^{k-3}\bar\gamma\gamma\bar\gamma)_B}
       \ar[d]_-{\sd^k\beta}
   & B\ar[r]^\alpha
   & \tri\;D
   \\
     **[l]\direalize{C}
       \ar[ur]_-{f}
   &
   & **[l]\sd^kC
   \ar@{.>}[urr]_-{\psi},
   }
\end{equation}
  where $\alpha,\beta$ are $\DIAGRAM$-simplicial functions, $C$ is finite, and $D$ is a $\DIAGRAM$-cubical set.
  For each $k\gg 0$, there exists a $\DIAGRAM$-simplicial function $\psi$ such that the right side commutes and $\direalize{\psi}\homotopic f\shomeo^k_{C}$.
\end{thm}
\begin{proof} 
  Fix $k\gg 0$.  
  Let $f'$ be the $\DIAGRAM$-stream map
  $$f'=\ordinomorphism^{-4}_{\tri D}f\ordinomorphism^k_C:\;\direalize{\sd^kC}\;\ra\;\direalize{\sd^4\tri D}.$$
  
  Let $\mathscr{A}$ be the small category whose objects are all $\DIAGRAM$-simplicial functions $\sigma:A\ra\sd^kC$ such that $A(g)$ is representable for finitely many of the $g$ and empty otherwise and whose morphisms are all commutative triangles.  
  We take $(\mathscr{A}\times\DIAGRAM)'$ to be the full subcategory of $\mathscr{A}\times\DIAGRAM$ consisting of pairs $(\sigma:A\ra\sd^kC,g)$ with $A(g)\neq\varnothing$.
  Let $C'$ and $D'$ be the $(\mathscr{A}\times\DIAGRAM)'$-simplicial set and $(\mathscr{A}\times\DIAGRAM)'$-cubical set naturally sending each pair $(\sigma:A\ra\sd^kC,g)$ respectively to $A(g)$ and $D(g)$ and $f''$ the $(\mathscr{A}\times\DIAGRAM)'$-stream map $\direalize{C'}\ra\direalize{\sd^4\tri D'}$ such that $f''_{(\sigma,g)}=f'_g\direalize{\sigma_g}$.  
  
  The set $f''_{(\sigma,g)}|C'(\sigma,g)|$ lies in an open star of a vertex of $\sd^4\tri D'(g)$ for each $(\mathscr{A}\times\DIAGRAM)'$-object $(\sigma,g)$ by $k\gg 0$ and $C$ compact.
  Thus there exist functor $R:(\mathscr{A}\times\DIAGRAM)'\ra\BOX$, $(\mathscr{A}\times\DIAGRAM)'$-cubical function $\iota:\BOX[R-]\ra\sd^2\tri D'$, and $(\mathscr{A}\times\DIAGRAM)'$-stream map $f''':\direalize{C'}\ra\direalize{\BOX[R-]}$ such that $\direalize{(\gamma\bar\gamma)_{\tri D'}(\tri\iota)}f'''=\direalize{(\gamma\bar\gamma)^2_{\tri D'}}f''$ [Lemma \ref{lem:bamfl} and Proposition \ref{prop:sd.cd.tri}].  
  The function
  $$\phi_{(\sigma:A\ra\sd^kC,g)}:[\dim A(g)]\ra R(\sigma,g),$$
  natural in $(\mathscr{A}\times\DIAGRAM)'$-objects $(\sigma,g)$ and defined by
  $$\phi_{(\sigma:A\ra\sd^kC,g)}(v)=\min\supp_{|-|}(f'''_{(\sigma,g)}|v|,|\tri\BOX[R(\sigma,g)]|)$$
  is monotone [Lemma \ref{lem:recover.orientations}].
  Thus $\sn\phi_{(\sigma,g)}$ defines a simplicial function
  $$A(g)\ra\sn R(\sigma,g)=\tri\BOX[R(\sigma,g)]$$
  natural in $(\mathscr{A}\times\DIAGRAM)'$-objects $(\sigma,g)$; hence $\sn\phi$ defines an $(\mathscr{A}\times\DIAGRAM)'$-simplicial function
  $$\sn\phi:C'\ra\tri\BOX[R-].$$
  The $(\mathscr{A}\times\DIAGRAM)'$-simplicial function
  $$(\gamma\bar\gamma)_{\tri D'}(\tri \iota)(\sn\phi):C'\ra\tri D'$$
  induces a $\DIAGRAM$-simplicial function
  $$\psi:\sd^kC\ra\tri D$$
  by taking colimits.  
  Observe $f'''\homotopic\direalize{\sn\phi}$ [Lemma \ref{lem:primordial.homotopy}], hence
  $$\direalize{(\gamma\bar\gamma)^2_{\tri D'}}f''=\direalize{(\gamma\bar\gamma)_{\tri D'}(\tri \iota)}f'''\homotopic\direalize{(\gamma\bar\gamma)_{\tri D'}(\tri \iota)\sn\phi},$$
  hence $\direalize{(\gamma\bar\gamma)^2_{\tri D}}f'\homotopic\direalize{\psi}$ by taking colimits and hence
  $$f\shomeo^k_C\homotopic\direalize{\psi}$$
  by $f\ordinomorphism^k_C\homotopic\direalize{(\gamma\bar\gamma)_{\tri D}^2}f'$ [Proposition \ref{prop:sd.natural.approx}].
  
  To show that the right side in (\ref{eqn:approx}) commutes, it suffices to consider the case $\beta_g$ object-wise inclusion and $B(g)$ empty or representable for each $g$ by naturality. 
  Hence it suffices to show that the $[0]$-component of the $g$-component of the right side in (\ref{eqn:approx}) commutes for each $g$ because simplicial functions from representable simplicial sets are determined by their $[0]$-components.
  For all $g$ and $v\in(\sd^kB(g))[0]$, 
  \begin{align*}
       \psi_g(\sd^k\beta_g)\langle v\rangle
    =& ((\gamma\bar\gamma)_{\tri D(g)}(\tri\iota_g))\langle\min\supp_{|-|}(f'''_{(v_*,g)}|v|,\tri\BOX[R(v_*,g)])\rangle\\
    =& \langle\min\supp_{|-|}(|(\gamma\bar\gamma)^2_{\tri D(g)}|f'_{(v_*,g)}|v|,\tri D(g))\rangle\\
    =& \langle\min\supp_{|-|}(|(\gamma\bar\gamma)^2_{\tri D(g)}|\ordinomorphism^{-4}_{\tri D(g)}|\alpha_g|\ordinomorphism^k_{C(g)}|v|,\tri D(g))\rangle\\
    =& \langle\min\supp_{|-|}(|(\gamma\bar\gamma)^2_{\tri D(g)}|\ordinomorphism^{k-4}_{\sd^4\tri D(g)}|(\sd^k\alpha)\langle v\rangle|,\tri D(g))\rangle\\
    =& \langle\min\supp_{|-|}(\ordinomorphism^{k-4}_{\tri D(g)}|((\gamma\bar\gamma)^2_{\sd^{k-4}\tri D(g)}(\sd^k\alpha))\langle v\rangle|,\tri D(g))\rangle\\
    =& \langle\min\supp_{\sd^{k-4}}((\gamma\bar\gamma)^2_{\sd^{k-4}\tri D(g)}(\sd^k\alpha)\langle v\rangle,\tri D(g))\rangle\\
    =& (\gamma^{k-4}(\gamma\bar\gamma)^2)_{\tri D(g)}(\sd^k\alpha)\langle v\rangle\\
    =& (\gamma^{k-3}\bar\gamma\gamma\bar\gamma)_{\tri D(g)}(\sd^k\alpha)\langle v\rangle\\
    =& \alpha(\gamma^{k-3}\bar\gamma\gamma\bar\gamma)_{B(g)}\langle v\rangle.
  \end{align*}
\end{proof}

\begin{cor}
  \label{cor:cd.approx}
  Consider a commutative diagram on the left side of
\begin{equation}
   \label{eqn:cd.approx}
   \xymatrix@C=3pc{
     **[l]\direalize{B}
       \ar[d]_-{\direalize{\;\beta\;}}
       \ar[r]^-{\direalize{\;\alpha\;}}
   & **[r]\direalize{\qua\;D}
   & **[l]\cd^k B
       \ar[r]^-{(\gamma^{k-3}\bar\gamma\gamma\bar\gamma)_B}
       \ar[d]_-{\cd^k\beta}
   & B\ar[r]^-{\alpha}
   & \qua D
   \\
     **[l]\direalize{C}
       \ar[ur]_-{f}
   &
   & **[l]\cd^kC
   \ar@{.>}[urr]_-{\psi},
   }
\end{equation}
  where $\alpha,\beta$ are $\DIAGRAM$-cubical functions, $C$ is finite, and $D$ is a $\DIAGRAM$-simplicial set.
  For each $k\gg 0$, there exists a $\DIAGRAM$-cubical function $\psi$ such that the right side commutes and $\direalize{\psi}\homotopic f\chomeo^k_{C}$.
\end{cor}
\begin{proof}
  Fix $k\gg 0$.
  Let $\tr=\tri$, $\qu=\qua$, $\epsilon$ and $\eta$ be the respective counit and unit of the adjunction $\tr\vdash\qu$.
  
  There exists a $\DIAGRAM$-simplicial function $\psi'$ such that $\direalize{\psi'}\homotopic f\shomeo^k_C$ and the diagram
  \begin{equation}
   \label{eqn:tri.cd.approx}
   \xymatrix@C=4pc{
     **[l]\tr\cd^kB
       \ar[r]^-{\tr(\gamma^{k-3}\bar\gamma\gamma\bar\gamma)_B}
       \ar[d]_-{\tr(\cd^k\beta)}
   & \tr B\ar[r]^{\tr\alpha}
   & \tr\;\qu\;D
   \\
     **[l]\tr\cd^kC
       \ar@{.>}[urr]_-{\psi'},
   }
  \end{equation}
  commutes [Theorem \ref{thm:sd.approx} and Proposition \ref{prop:sd.cd.tri}].
  Let $\psi$ be the composite
  \begin{equation*}
    \sd^kC\xra{\eta_{\sd^kC}}\qu\tr\sd^kC\xra{\qu\psi'}\qu\tr\qu D\xra{\qu\epsilon_D}\qu D.
  \end{equation*}
  The right side of (\ref{eqn:cd.approx}) commutes by an application of (\ref{eqn:tri.cd.approx}), naturality and a zig-zag identity.  

  To see $\direalize{\psi'}\homotopic\direalize{\tr\psi}$ and hence $f\homotopic\direalize{\psi}$, consider the diagram
  \begin{equation*}
    \xymatrix@C=4pc{
    & **[r]\tr\qu\tr\cd^kC\ar[r]^-{\tr\qu\psi'}\ar[d]|-{\epsilon_{\tr\cd^kC}}
    & (\tr\qu)^2D\ar[r]^-{\tr\qu\epsilon_D}\ar[d]|-{\epsilon_{\tr\qu\;D}}
    & \tr\qu D
    \\
      **[l]\tr\cd^kC\ar[ur]^-{\tr\eta_{\sd^kC}}\ar[r]_-{\tr\id_{\cd^kC}}
     & **[l]\tr\cd^kC\ar[r]_-{\psi'}
    & \tr\qu D.\ar[ur]_-{\id_{\tr\qu D}}
    }
  \end{equation*}
  The left triangle commutes by a zig-zag identity. 
  The middle square commutes by naturality.
  After applying $\direalize{-}$, the right triangle commutes up to $\homotopic$ by a zig-zag identity and $\direalize{\tr\;\eta_{\qu D}}$ an inverse to $\direalize{\epsilon_{\tr\qu D}}$ up to $\homotopic$ [Lemma \ref{lem:tri.equivalence}].  
  Thus $\direalize{\psi'}\homotopic\direalize{\tr\;\psi}$.  
\end{proof}

\appendix
\addcontentsline{toc}{section}{Appendix}
\addtocontents{toc}{\protect\setcounter{tocdepth}{0}}

\section{Preorders}\label{sec:preorders}

Fix a set $X$. 
Generalizing a function $X\ra X$, a \textit{relation $\vartriangleright$ on $X$} is the data of $X$, the set on which $\vartriangleright$ is defined, and its \textit{graph} $\graph{\vartriangleright}$, a subset of $X\times X$.
For each relation $\vartriangleright$ on $X$, we write $x\vartriangleright y$ whenever $(x,y)\in\graph{\vartriangleright}$.  
A \textit{preorder $\leqslant_X$ on $X$} is a relation on $X$ such that $x\leqslant_Xx$ for all $x\in X$ and $x\leqslant_Xz$ whenever $x\leqslant_Xy$ and $y\leqslant_Xz$.  
A \textit{preordered set} is a set $X$ implicitly equipped with a preorder, which we often write as $\leqslant_X$.  
An \textit{infima} of a subset $A$ of a preordered set $X$ is an element $y\in X$ such that $y\leqslant_Xa$ for all $a\in A$ and $x\leqslant_Xy$ whenever $x\leqslant_Xa$ for all $a\in A$.  
We dually define \textit{suprema} of subsets of preordered sets.  
An element $m$ of a preordered set $X$ is a \textit{minimum of $X$} if $m\leqslant_Xx$ for all $x\in X$.  
We dually define a \textit{maximum} of a preordered set. 

\begin{eg}
  The minima and maxima of $[n]$ are $0$ and $n$, respectively.
\end{eg}

\begin{eg}
  The minima and maxima of $\boxobj{n}$ are, respectively
  $$(0,\ldots,0),\quad (1,\ldots,1).$$
\end{eg}

A function $f:X\ra Y$ from a preordered set $X$ to a preordered set $Y$ is \textit{monotone} if $f(x)\leqslant_Xf(y)$ whenever $x\leqslant_Xy$.  
A monotone function $f:X\ra Y$ is \textit{full} if $x\leqslant_Xy$ whenever $f(x)\leqslant_Xf(y)$.

\begin{eg}
  The isomorphisms in $\PREORDEREDSETS$ are the full monotone bijections.
\end{eg}

A monotone function $f:X\ra Y$ between preordered sets having minima and maxima \textit{preserves extrema} if $f$ sends minima to minima and maxima to maxima.  
Preordered sets and monotone functions form a complete, cocomplete, and Cartesian closed category $\PREORDEREDSETS$.
The mapping preordered set $Y^X$ is the set of monotone functions $X\ra Y$ equipped with the pointwise preorder induced from $\leqslant_Y$.  

\begin{eg}
In particular, there exists a bijection
$$X^{[1]}\cong\graph{\leqslant_X}$$
of underlying sets, natural in preordered sets $X$.
\end{eg}

A \textit{semilattice} is a set $X$ equipped with a commutative, associative, and idempotent binary multiplication $X\times X\ra X$, which we write as $\meet_X$.
A \textit{semilattice homomorphism} between semilattices is a function preserving the multiplications.  
We can regard semilattices as preordered sets as follows.  

\begin{lem}
  For each preordered set $X$, the following are equivalent:
  \begin{enumerate}
    \item The set $X$ is a semilattice such that $x\leqslant_Xy$ if and only if $x\meet_Xy=x$.  
    \item Each subset $\{y,z\}\subset X$ has a unique \textit{infimum}.
  \end{enumerate}
  If these equivalent conditions are satisfied, then $\meet_X$ describes taking binary infima.
\end{lem}

A \textit{lattice} $X$ is a set $X$ equipped with a pair of commutative and associative binary multiplications $X\times X\ra X$, which we often write as $\meet_X,\join_X$ and respectively call the \textit{meet} and \textit{join} operators of $X$, such that $x\join_Xx=x\meet_Xx=x$ for all $x\in X$ and $x\join_X(x\meet_Xy)=x\meet_X(x\join_Xy)=x$ for all $x,y\in X$.  
We can regard lattices as preordered sets equipped with dual semilattice structures as follows.

\begin{lem}
  \label{lem:lattices}
  For each preordered set $X$, the following are equivalent:
  \begin{enumerate}
    \item The set $X$ is a lattice such that $x\leqslant_Xy$ if and only if $x\meet_Xy=x$.  
    \item The set $X$ is a lattice such that $x\leqslant_Xy$ if and only if $x\join_Xy=y$.  
    \item Each subset $\{y,z\}\subset X$ has a unique \textit{infimum} and a unique \textit{supremum}.
  \end{enumerate}
  If these equivalent conditions are satisfied, then the $\join_X$ describes taking binary suprema and $\meet_X$ describes taking binary infima.
\end{lem}

A function $\psi:X\ra Y$ from a lattice $X$ to a lattice $Y$ is a \textit{homomorphism} if $\psi$ preserves the multiplications.
It follows from Lemma \ref{lem:lattices} that every homomorphism of lattices defines a monotone function, full if injective, of preordered sets.

A \textit{topological lattice} is a lattice $X$ topologized so that its lattice operations $\meet_X,\join_X:X\times X\ra X$ are continuous.  
A (\textit{real}) \textit{lattice-ordered topological vector space} is a (real) topological vector space $V$ equipped with a preorder $\leqslant_V$ turning $V$ into a topological lattice such that $\alpha a+\beta b\leqslant_{V}\alpha a'+\beta b'$ for all $a\leqslant_Va'$, $b\leqslant_Vb$, and $\alpha,\beta\in[0,\infty)$. 
An example is the standard topological vector space $\R^n$, having as its meet and join operations coordinate-wise min and max functions.

\section{Coends}\label{sec:coends}
For each small category $\DIAGRAM$, complete category $\GENERIC$, and functor $F:\OP\DIAGRAM\times\DIAGRAM\ra\GENERIC$, the \textit{coend} $\int_{\DIAGRAM}^{g}F(g,g)$ of $F$ is defined by the coequalizer diagram
\begin{equation*}
  \xymatrix@C=4pc{
    \coprod_{\gamma:g'\ra g''}F(g'',g')\ar@<.7ex>[r]\ar@<-.7ex>[r]
  & \coprod_{g}F(g,g)\ar[r]
  & \int_{\DIAGRAM}^gF(g,g) 
}
\end{equation*}
where the first coproduct is taken over all $\GENERIC$-morphisms $\gamma:g'\ra g''$, the second coproduct is taken over all $\GENERIC$-objects $g$, the top left arrow is induced from $\GENERIC$-morphisms of the form $F(g'',\gamma)$, and the bottom left arrow is induced from $\GENERIC$-morphisms of the form $F(\gamma,g')$.  
An example is the geometric realization
$$|B|=\int_{\DEL}^{[n]}B[n]\cdot\cosimplicial[n]$$
of a simplicial set $B$, where $\cosimplicial$ is the functor $\DEL\ra\SPACES$ naturally assigning to each non-empty finite ordinal $[n]$ the topological $n$-simplex $\cosimplicial[n]$.

\end{document}